\documentclass[a4paper,reqno, 11pt]{amsart}

\usepackage{amsmath}
\usepackage{paralist}
\usepackage{graphics}
\usepackage{epsfig}
\usepackage{amsfonts}
\usepackage{amssymb}
\usepackage{color}
\usepackage{hyperref}
\usepackage{epstopdf}
\usepackage{bm}

\allowdisplaybreaks

\newtheorem{theorem}{Theorem}[section]
\newtheorem{lemma}[theorem]{Lemma}

\def\ifl{\iffalse }

\def\bc{\begin{center}}       \def\ec{\end{center}}
\def\ba{\begin{array}}        \def\ea{\end{array}}
\def\be{\begin{equation}}     \def\ee{\end{equation}}
\def\bea{\begin{eqnarray}}    \def\eea{\end{eqnarray}}
\def\beaa{\begin{eqnarray*}}  \def\eeaa{\end{eqnarray*}}

\numberwithin{equation}{section}

\newtheorem{remark}[theorem]{Remark}
\numberwithin{equation}{section}

\begin{document}
\author{Huicong Li}
\address{School of Mathematics, Sun Yat-sen University, Guangzhou, 510275, Guangdong Province, China}
\email{lihuicong@mail.sysu.edu.cn}
\author{Tian Xiang$^\dagger$}
\address{Institute for Mathematical Sciences, Renmin University of China, Bejing, 100872, China}
\email{txiang@ruc.edu.cn}
\thanks{$^\dagger$ Corresponding author.}

\title[On a cross-diffusive SIS epidemic model with nonlinear incidence]{On a cross-diffusive SIS epidemic model with power-like nonlinear incidence}

\begin{abstract}
We study global existence, boundedness and convergence of nonnegative classical solutions of a Neumann initial-boundary value problem for the following cross-diffusive SIS (susceptible-infected-susceptible) epidemic model with power-like infection mechanism generalizing the standard mass action mechanism:
$$
\begin{cases}
S_t=d_S \Delta S+\chi\nabla\cdot( S \nabla I)-\beta S^ q I^ p +\gamma I,& x\in\Omega,\, t>0,\\[0.25cm]
I_t=d_I\Delta I+\beta  S^ q I^ p -(\gamma+\mu) I,& x\in\Omega,\, t>0,
\end{cases}
$$
in a bounded smooth domain $\Omega\subset\mathbb{R}^n \ (n\geq 1)$. The incidence of the form $\beta S^q I^p$ with $p, \, q>0$ is a natural extension of the classical mass action type $\beta SI$, and the cross-diffusive term $\chi \nabla \cdot (S \nabla I)$ with $\chi\geq 0$ describes the effect that the susceptible individuals tend to move away from higher concentration of infected populations. Global existence and boundedness of classical solutions are established in certain parameter ranges, and threshold/non-threshold long-time behaviors of global bounded solutions are also detected. Our findings significantly  improve and extend  previous related studies.
\end{abstract}

% 35A01 existence problems: global existence, local existence, non-existence
% 35B40 Asymptotic behavior of solutions
% 35K57 Reaction-diffusion equations
% 35Q92 PDEs in connection with biology and other natural sciences
% 92D30 Epidemiology


\maketitle

\numberwithin{equation}{section}

\section{Introduction}

In the theory of epidemiology, incidence is the occurrence of new cases of a disease via contact between infected individuals and susceptible ones within a specified period of time, and it is usually characterized by the infection mechanism in the modeling of infectious diseases. In the seminal work of Kermack and McKendrick \cite{K-M}, the authors studied a susceptible-infected-recovered compartmental model with bilinear incidence rate $\beta SI$, where $S$ and $I$, respectively, are the numbers of susceptible and infected populations, and the positive constant $\beta$ measures the disease transmission rate. However, as argued in \cite{deJong, McCallum}, this type of bilinear infection mechanism, describing the simple mass action law, has its own shortcomings and may not fit the field data well. Thus, various different types of transmission mechanism were proposed and investigated, such as the general nonlinear transmission rate of the form $\beta S^q I^p$ with $p$, $q>0$ \cite{Hethcote, Koro, Lei1, Lei2, MLi, Liu1, Liu2, PW19, Severo}, the standard incidence type of $\beta \frac{SI}{S+I}$ \cite{Allen1, Allen2, Cui1, CLPZ, Cui2, Kuto, LiLiTong, LiPeng, LiPengWang-JDE, Peng, Peng-Liu, PengYi, PengZhao}, the saturating incidence rate $\beta\frac{SI}{1+mI}$ with parameter $m>0$ \cite{Capasso-Serio, Cui21, Hu, Liu-Cui, Liu, Raja, Wang, Xu}, etc; see also references therein for many other types of nonlinear infection mechanisms.

In this paper, inspired mainly from \cite{DW16, LPX18, PW19}, we are interested in the dynamical properties of solutions to the following homogeneous Neumann IBVP  (initial-boundary value problem)  of  (possibly) cross-diffusive SIS epidemic model with generalized  mass action infection mechanism:
\be\label{SIS-mass-C}
\begin{cases}
S_t=d_S \Delta S+\chi\nabla\cdot( S \nabla I)-\beta S^ q I^ p +\gamma I,& x\in\Omega,\, t>0,\\[0.25cm]
I_t=d_I\Delta I+\beta  S^ q I^ p -(\gamma+\mu) I,& x\in\Omega, \, t>0,\\[0.25cm]
\frac{\partial S}{\partial \nu}=\frac{\partial I}{\partial \nu}=0, & x\in \partial \Omega, \, t>0,\\[0.25cm]
S(x,0)=S_0(x),  \ \   I(x,0)= I_0(x), & x\in \Omega.
\end{cases}
\ee
Here and below, $\Omega\subset\mathbb{R}^n~(n\geq 1)$ is a bounded domain with smooth boundary $\partial\Omega$. The unknown functions
$S=S(x,t)$ and $I=I(x,t)$, respectively, denote the population density of susceptible and infected individuals at location $x$ and time $t$; $d_S$ and $d_I$ are positive constants measuring the random mobility of susceptible and infected populations respectively; the featured cross-diffusion term $\chi \nabla \cdot (S\nabla I)$ describes the biased  movement (widely known as repulsive chemotaxis effect in chemotactic models) of those \lq\lq smart" susceptible individuals  who  tend to stay away from regions of  higher concentration of infected individuals  with nonnegative constant $\chi$ measuring the strength  of such effect; positive constants $\beta$ and $\gamma$, respectively, stand for the disease transmission rate and recovery rate, whereas the parameter $\mu \geq 0$ is the disease-induced death rate. The homogeneous Neumann boundary conditions imply there is no population flux across the boundary $\partial \Omega$. As explained in \cite{PW19}, we assume within this paper that the initial datum $(S_0,I_0)$ satisfies
 \be\label{initial data-req}
\begin{cases}
S_0\geq 0, \   I_0\geq, \, \not\equiv 0, \   (S_0,I_0)\in C^0(\overline{\Omega})\times W^{1,r_0}(\Omega) \text{ for  some } r_0>\max\{2,n\}, \\[0.2cm]
\inf_{\Omega}S_0>0,  \ \  \text{ if  } 0<q<1, \\[0.2cm]
\inf_{\Omega}I_0>0, \ \  \text{ if  } 0<p<1.
\end{cases}
\ee

The PDE model \eqref{SIS-mass-C} with $\chi=0$, $q=p=1$ and $\mu=0$ was proposed and studied by Deng and Wu \cite{DW16}, wherein they established the global attractiveness of the DFE (disease-free equilibrium) and EE (endemic equilibrium) in some special cases, and the existence (and uniqueness) of EE in a heterogeneous environment was also discussed when the basic reproduction number is greater than one. The asymptotic behavior of EE with respect to diffusion rates $d_S$ and $d_I$ was investigated later in \cite{Wu-Zou}; see also further development in \cite{Salako, Wen, LiPengWang}. Recently, Peng and Wu \cite{PW19} studied the reaction-diffusion model \eqref{SIS-mass-C} without cross-diffusion ($\chi=0$) and they obtained the global existence and boundedness of classical solutions for any $p, \, q>0$ by using a duality argument, and then,  they derived long time behaviors of solutions for  $\mu>0$ and uniform persistence for $\mu=0$ upon establishing the ultimately uniform boundedness. More dynamics on many special cases of the reaction-diffusion model \eqref{SIS-mass-C} without cross-diffusion can be found in the introduction of \cite{PW19}.

Due to the strong coupling in the (highest) 2nd order term, in study of chemotaxis systems, it is well-known that the chemotactic cross-diffusion $\chi \nabla \cdot (S \nabla I)$ has  a strong tendency towards driving the solutions of the underlying models to blow-up in finite or infinite time; see \cite{BBTW15, Win-blow1, Win-blow2} for example. Compared to the extensive literature on various chemotactic systems, it seems that epidemic models with cross-diffusive terms were much less studied. To the best of our knowledge, we are only aware of very few papers \cite{Wang2, LPX18, Wang1} dealing with such mathematical models (see \cite{BPTW19,  Tao-Win21, Win19}, etc, for May-Nowak type cross-diffusive models for virus infection). Moreover, it should be noted that in these works, the infection mechanism is taken to be the standard incidence type of the form $\beta\frac{SI}{S+I}$ (in SIS models) or $\beta \frac{SI}{S+I+R}$ (in SIRS models). From the mathematical point of view, such frequency-dependent transmission term grows at most linearly in $I$, and thus one can take full advantage of this property to derive not only the $L^\infty$-bound of $I$ from its $L^1$-bound, but also the $L^\infty$-bound of $\nabla I$. As a result, global existence and boundedness of solutions can be established in \cite{Wang2, LPX18, Wang1} in any spatial dimensions, regardless of the magnitude of $\chi$. Nevertheless, in our cross-diffusive model \eqref{SIS-mass-C}  with power-like nonlinear incidence $\beta S^q I^p$, the above strategy is no longer available, and the duality argument  used in \cite{PW19} also fails because of the presence of the chemotactic cross-diffusion. Even when $p=q=1$, performing similar yet lengthy computations as done in \cite{TW12-JDE, Win12-CPDE, Xiang16-ppt} aimed  to cancel out the chemotaxis involving term from \eqref{SIS-mass-C}, we find that
\be\label{S+gradI/I-id}
\begin{split}
&\frac{d}{dt}\int_\Omega \left[2\beta S(\ln S-1)+2\chi \left|\nabla I^\frac{1}{2}\right|^2\right] + 8\beta d_S\int_\Omega \left|\nabla S^\frac{1}{2}\right|^2\\
&\quad + 4(\gamma+\mu)\chi\int_\Omega \left|\nabla I^\frac{1}{2}\right|^2 + 2d_I\chi\int_\Omega I \left|D^2\ln I\right|^2 \\
&=-2\beta^2\int_\Omega SI \ln S+2\beta\gamma\int_\Omega I \ln S+d_I\chi\int_{\partial\Omega} \frac{1}{I} \frac{\partial \left|\nabla I\right|^2}{\partial\nu} + \beta\chi \int_\Omega \frac{S}{I}\left|\nabla I\right|^2.
\end{split}
\ee
Thus,  the chemotactic cross-diffusion induced term has been offset but the last term seems uncontrollable in terms of the dissipative terms on the left-hand-side of \eqref{S+gradI/I-id}. Even facing such challenges (strong -no saturation- chemotactic cross-diffusion in the second order term  and strong nonlinear coupling in the zeroth order term), we shall employ subtle energy estimates and/or semi-group arguments to bootstrap the easily obtained a priori $L^1$-boundedness  to $L^\infty$-boundedness of classical solutions to the IBVP \eqref{SIS-mass-C} under either
$$
np+ \left(n-2\right)^+q<n+\min\left\{n, \  2\right\}, \ \ \  |\chi|<\chi_0 \  \text{with} \ \chi_0 \  \text{defined by}\  \eqref{chi-max},
$$
or
\be\label{thm-con+0}
 q<\frac{1}{n+1}\quad \mbox{and} \quad \ p+(n+1)q < 1+ \min\left\{1,\frac{2}{n}\right\},
\ee
or
\be\label{kappa-lambda-con0}
\begin{cases}
 10 q+4 p <15, \ \ \ \  q+ p <3,  & \text{ if } n=1, \\[0.25cm]
  3 q+ p <3, \ \ \ \  q+ p <2,  & \text{ if } n=2.
  \end{cases}
\ee
\textbf{We here underline that, under either  \eqref{thm-con+0} or \eqref{kappa-lambda-con0}, boundedness of classical solutions to the IBVP \eqref{SIS-mass-C} is ensured without any smallness assumption on the initial data; see Theorems \ref{glob bdd-semi1}, \ref{th+-semi-} and \ref{glob bdd3} for details}.

Going beyond boundedness, the second part of this project is devoted to  long-time behaviors of global bounded solutions to  the IBVP  \eqref{SIS-mass-C}. \textbf{In the presence  of  mortality, i.e.,  $\mu>0$,  it is shown in Theorem \ref{thm-lt1}  that the usual threshold dynamics for epidemic models in terms of a so-called basic reproduction number do not exist}. More specifically,  as $t\rightarrow \infty$, any bounded global solution $(S,I)$ converges   uniformly on $\overline{\Omega}$  according to
\be
\nonumber
%\label{S-con0}
\left(S(\cdot,t),I(\cdot,t)\right)\rightarrow \begin{cases}(0,0),  &\text{ if }  0<p<1, \\[0.2cm]
(S_\infty,0),  &\text{ if }  p\geq 1.
\end{cases}
\ee
Furthermore, $S_\infty\leq \left(\frac{\gamma + \mu }{\beta}\right)^{\frac{1}{q}}$ if $p=1$ and $(S, I)\rightarrow (S_\infty,0)$ exponentially if $p>1$. Here, $S_\infty$ is a positive constant  and is given implicitly by
\be\label{S*-def0}
S_\infty=|\Omega|^{-1}\left(\int_\Omega \left(S_0+I_0\right)-\mu\int_0^\infty\int_\Omega I\right).
\ee
 Comparing with \cite[Theorem 2.4]{PW19}, we see that the dynamics of the cross-diffusive system stay  the same as that of the system without cross-diffusion, indicating  the effect of chemotaxis is overbalanced by mortality in respective of long time behavior. \textbf{Here, we would like to add that, we use a different approach to derive such refined convergence by identifying $S_\infty$ in \eqref{S*-def0} and showing that the convergence is exponential when $p>1$.}

\textbf{In the absence  of  mortality and of cross-diffusion, i.e.,  $\mu = \chi =0$, allowing $\beta$ and $\gamma$ to be spatially-dependent, via the basic reproduction number, threshold dynamics of global solutions to  \eqref{SIS-mass-C} with  $p=q=1$  was shown  in \cite{DW16} in two special cases,  and only  uniform persistence of global solutions is shown in \cite{PW19}. In the above scenario, we strengthen the results in \cite{PW19}}.  More specifically,  for $0<p<1$, we  provide non-threshold dynamics of global and bounded solutions; that is, we prove the global attractivity of the unique constant or non-constant  EE:
 \begin{itemize}
 \item[(C1)] Assume that $0<p<1$ and $\gamma(x)\equiv r\beta(x)$ on $\overline{\Omega}$ for some $r>0$. Then,  as $t\to\infty$,
\be
\left(S(\cdot,t),I(\cdot,t)\right) \to (S^*,I^*)
\nonumber
\ee
uniformly on $\overline\Omega$, where $(S^*, I^*)$ is the constant EE defined by \eqref{S-p1}.
  \item[(C2)] When  $\frac{\gamma}{\beta}\not\equiv \text{const}$, assume that $0<p<1$ and $d_S=d_I$. Then,  as $t\to\infty$,
\be
\left(S(\cdot,t),I(\cdot,t)\right) \to (\tilde S,  \tilde I)
\nonumber
\ee
uniformly on $\overline\Omega$, where $(\tilde S,  \tilde I)$ is the positive non-constant  EE of  \eqref{SIS-mass-C} with $\chi=0$.
 \end{itemize}
 While in the case of $p=1$, we establish threshold dynamics in terms of the basic reproduction number $\mathcal{R}_0$ defined by \eqref{R0}:
  \begin{itemize}
 \item[(C3)] Assume that $p=1$ and $\gamma(x)\equiv r\beta(x)$ on $\overline{\Omega}$ for some $r>0$. Then,  as $t\to\infty$,
\begin{align*}
\left(S(\cdot,t),I(\cdot,t)\right)\rightarrow \begin{cases} \left(\frac{N}{|\Omega|}, 0\right), &\text{if } \mathcal{R}_0\leq 1\Longleftrightarrow \frac{N}{|\Omega|}\leq r^\frac{1}{q},\\[0.25cm]
\left(r^\frac{1}{q}, \frac{N}{|\Omega|}-r^\frac{1}{q}\right), &\text{if } \mathcal{R}_0> 1\Longleftrightarrow \frac{N}{|\Omega|}> r^\frac{1}{q},
\end{cases}
\end{align*}
uniformly on $\overline\Omega$, where $N=\int_\Omega (S_0+I_0)$ is the conserved total population.
  \item[(C4)]When  $\frac{\gamma}{\beta}\not\equiv \text{const}$, assume that $p=1$ and $d_S=d_I$. Then,  as $t\to\infty$,
\begin{align*}
\left(S(\cdot,t),I(\cdot,t)\right)\rightarrow \begin{cases} \left(\frac{N}{|\Omega|}, 0\right), &\text{if } \mathcal{R}_0\leq 1,\\[0.25cm]
(\tilde S,  \tilde I), &\text{if } \mathcal{R}_0> 1,
\end{cases}
\end{align*}
uniformly on $\overline\Omega$, where $(\tilde S,  \tilde I)$ is the positive non-constant  EE of  \eqref{SIS-mass-C} with $\chi=0$.
 \end{itemize}
 See Theorems \ref{thm-glo1}, \ref{thm-glo2}, \ref{thm-glo3} and \ref{thm-glo4} for more details.
Particularly, \textbf{in the special case that both $\beta$ and $\gamma$ are positive constants, for all $0< p \leq 1$ and $q>0$, we provide a complete description of the long-time behavior of solutions to system \eqref{SIS-mass-C} with $\chi=0$}.
Here, we would like to mention that the global asymptotic stability of DFE or EE in general case seems to be rather challenging, which remains largely open even for much simpler SIS epidemic models than  \eqref{SIS-mass-C}. We also mention, besides (C1)-(C4),   many of our other results still hold if $\beta$ and $\gamma$ are spatial-temporally dependent, as long as they are H\"{o}lder continuous and bounded from below and above by positive constants.

The remainder of this project is organized as follows. In Section 2, we present some preliminary results and the local well-posedness of our model. Section 3 is devoted to the global existence and boundedness of solutions where we prove a couple of results based on different approaches. Finally, in the last section, we shall investigate the large-time convergence of bounded and global classical solutions.

For convenience and simplicity, in the sequel, the symbols $C$, $C_i(i=1,2,\cdots)$ or $C_\epsilon$  will denote  generic positive constants which are independent of time and may vary line-by-line. The spatial  integration symbol $dx$ will also  be omitted when no confusion could arise.

\section{Preliminary and local well-posednesss}

Let us start with the elementary  Young's inequality with epsilon: Let $r$  and $s$ be two given positive numbers with $\frac{1}{r}+\frac{1}{s}=1$. Then, for any  $\epsilon>0$,
$$
ab\leq \epsilon a^r+\frac{b^s}{(\epsilon r)^\frac{s}{r}},  \quad \quad \forall\, a,b\geq 0.
$$

In several places, we need the well-known  Gagliardo-Nirenberg interpolation inequality, we state it  here for the convenience of reference.
\begin{lemma}[Gagliardo-Nirenberg  interpolation inequality \cite{Fried,Nirenberg66}]
\label{GNI}
Let $\Omega\subset\mathbb{R}^n \ (n\geq 1)$ be a bounded and  smooth domain. Let $l$ and $k$ be integers satisfying $0\leq l<k$, and let $1\leq s,\, r\leq \infty$,   $0<m\leq \infty$ and $\frac{l}{k}\leq a\leq 1$ such that
$$
\frac{1}{m}-\frac{l}{n}=a\left(\frac{1}{s}-\frac{k}{n}\right)+(1-a)\frac{1}{r}.
$$
Then, for any   $f\in W^{k,s}(\Omega)\cap L^r(\Omega)$, there exist a positive constant $C$ depending only on $\Omega,k,r,s$ and $n$ such that
\begin{equation}\label{G-N-I}
\|D^l f\|_{L^m}\leq C \left(\|D^k f\|_{L^s}^a\|f\|_{L^r}^{1-a}+\|f\|_{L^r}\right),
\end{equation}
with the following exception: if $1<s<\infty$ and $k-l-\frac{n}{s}$ is a nonnegative integer, then $\eqref{G-N-I}$ holds only  for $a$ satisfying $\frac{l}{k}\leq a<1$.
\end{lemma}

Henceforth,  for notational convenience, we write $\|f\|_{L^p}$ for the usual $L^p(\Omega)$-norm of $f$.

In our subsequent analysis, we shall  apply the commonly used smoothing $L^r$-$L^s$ estimates  of the Neumann heat semi-group $\left(e^{td\Delta}\right)_{t\geq0}$ on a bounded and smooth domain $\Omega$; see, for instance,  \cite{Cao15, LPX18, Win10-JDE}. Again, we list them here for direct reference.
\begin{lemma}
\label{semigroup}
For $d>0$, let $\left(e^{td\Delta}\right)_{t\geq0}$ be the Neumann heat semi-group and $ \lambda_1=: \lambda_1(d)>0$ be the first positive Neumann eigenvalue of $-d\Delta$ on $\Omega$. Then there exist some positive constants $k_i$ {\rm (}$i=1,2,3,4${\rm )} depending only on $d$ and $\Omega$ fulfilling the following smoothing $L^r$-$L^s$ type estimates:
\begin{enumerate}
\item[{\rm (i)}] If $1\leq s\leq r\leq \infty$, then, for any $f\in L^s(\Omega)$,
\begin{equation}
\nonumber
\left\|e^{td\Delta}f\right\|_{L^r}\leq k_1 \left(1+t^{-\frac{n}{2}\left(\frac{1}{s}-\frac{1}{r} \right)} \right)\left\|f\right\|_{L^s},\quad\forall\, t>0.
\end{equation}

\item[{\rm (ii)}] If $1\leq s\leq r\leq \infty$, then, for all $f\in L^s(\Omega)$,
\begin{equation}
\nonumber
\left\|\nabla e^{td\Delta}f\right\|_{L^r}\leq k_2 \left(1+t^{-\frac{1}{2}-\frac{n}{2}\left(\frac{1}{s}-\frac{1}{r} \right)} \right)e^{-\lambda _1t}\left\|f\right\|_{L^s},\quad\forall\, t>0.
\end{equation}

\item[{\rm (iii)}] If $2\leq s\leq r<\infty$, then,  for all $f\in W^{1,s}(\Omega)$,
\begin{equation}
\nonumber
\left\|\nabla e^{td\Delta}f\right\|_{L^r}\leq k_3 \left(1+t^{-\frac{n}{2}\left(\frac{1}{s}-\frac{1}{r} \right)} \right)e^{- \lambda_1t}\left\|\nabla f\right\|_{L^s},\quad\forall\, t>0.
\end{equation}

\item[{\rm (iv)}] If $1<s\leq r\leq\infty$, then, for all $f\in \left(L^s(\Omega)\right)^n$,
\begin{equation}
\nonumber
\left\|e^{td\Delta}\nabla \cdot f\right\|_{L^r}\leq k_4 \left(1+t^{-\frac{1}{2}-\frac{n}{2}\left(\frac{1}{s}-\frac{1}{r} \right)} \right)e^{- \lambda_1t}\left\| f\right\|_{L^s},\quad\forall\, t>0.
\end{equation}
\end{enumerate}
\end{lemma}

 Using variation-of-constants formula for $S$ and $I$,  applying these semi-group estimates as in Lemma \eqref{semigroup}, and then, invoking the Banach contraction mapping principle and standard parabolic regularity of quasilinear parabolic systems, one can readily derive the local-in-time  solvability and extensibility of the IBVP for the cross-diffusive SIS model \eqref{SIS-mass-C}; see similar reasonings in \cite{BBTW15, Tao10, Tao11, Win12-CPDE, WSW16}.

\begin{lemma}
\label{local-in-time}
Let  $\Omega\subset \mathbb{R}^n \ (n\geq 1)$ be a bounded domain with a smooth boundary, and suppose that the initial data $(S_0,I_0)$ satisfies \eqref{initial data-req}. Then there is a unique positive and classical maximal solution $(S, I)$ of the IBVP \eqref{SIS-mass-C}  on some maximal interval $[0, T_m)$ with $0<T_m \leq \infty$ such that
$$
(S, I)\in \left(C(\overline{\Omega}\times [0, T_m);W^{1,r_0}(\Omega))\cap C^{2,1}(\overline{\Omega}\times (0, T_m)\right)^2.$$
Furthermore, if  $T_m<\infty$, then
\begin{equation}
\|S(\cdot,t)\|_{L^\infty}+\|I(\cdot,t)\|_{W^{1,r_0}}\to\infty \mbox{ as } t\nearrow T_m.
\label{extend}
\end{equation}
On $[0, T_m)$, the $L^1$-norm of   $S+I$ is non-increasing and hence is uniformly bounded:
\be
\label{L1-id}
\int_\Omega (S+I) \leq  \int_\Omega (S_0+I_0).
\ee
\end{lemma}

\begin{proof} As remarked above, the statements concerning the local-in-time existence of  classical solutions to the IBVP \eqref{SIS-mass-C}  and the extensibility criterion \eqref{extend} are well-established. The positivity of $(S, I)$ follows simply from the strong maximum principle.  Thanks to the  homogeneous Neumann boundary conditions, adding  the $S$- and $I$-equation in \eqref{SIS-mass-C} and integrating by parts, we find that
\be\label{S+I-id}
\frac{d}{dt}\int_\Omega \left(S+I\right)=-\mu \int_\Omega I,
\ee
which, upon integration, simply yields \eqref{L1-id}.
\end{proof}
For any $\tau\in(0, T_m)$, we can shift the initial time $t=0$ to $t=\tau$ so that upon retreating $(u(\cdot, \tau), v(\cdot, \tau))$ as the new initial data, we will sometimes assume that $r_0=\infty$ in \eqref{initial data-req}.

\section{A priori estimates and global well-posedness}
In this section, unless otherwise specified, we will assume  that the basic conditions in the local existence of Lemma \ref{local-in-time} hold. With such specifications,  we shall aim to  derive   $(L^\infty, W^{1,r_0})$-type estimates  of $(S, I)$ under two set of conditions: (I) $\chi$ is small and $p,q$ are relatively less restrictive, and (II) $\chi$ is arbitrary and $p,q$ are relatively restrictive. The small-$\chi$ boundedness is mainly based on Neumann semi-group type estimates and the large-$\chi$ boundedness is based on mainly coupled energy estimate method.
\subsection{Small $\chi$-boundedness and global existence}
In this subsection, we renovate a loop argument   from \cite{BPTW19} to provide boundedness of classical solutions to \eqref{SIS-mass-C} with $\chi$  small and $p,q$  relatively less restrictive.

\begin{lemma}\label{Ibdd-lemma} Given  $T\in(0, T_m)$, and suppose there exists $M=M(T)>0$ such that
 \be\label{S-bdd-imply Ibdd-con}
 \|S(\cdot,t)\|_{L^\infty}\leq M, \quad \quad \forall\, t\in(0, T).
 \ee
Let
\begin{align*}
 p<1+\min\left\{1, \  \frac{2}{n}\right\}.
\label{2022-1}
\end{align*}
Then, for
\be\label{mq-range}
  m=\begin{cases}
  \max\left\{1,\,\frac{1}{p}\right\}, & \text{if }   p<2, \  n=1,  \\[0.2cm]
 \frac{1}{p}, & \text{if }   p< \frac{2}{n},  \ n\geq 2,  \\[0.2cm]
   \frac{n}{2-n\epsilon}, & \text{if  } \frac{2}{n} \leq p < 1+\frac{2}{n},\ n\geq 2,
\end{cases}
 \ee
 with $0< \epsilon<\min\left\{\frac{2}{n}, \ 1+\frac{2}{n}-p\right\}$,
one can find $K_1>0$ independent of $M$ such that
\be\label{S-bdd-imply Ibdd0}
L:=\sup_{t\in(0, T)}\|I(\cdot,t)\|_{L^\infty}\leq K_1\left(1+M^\frac{qm}{m+1-pm}\right)
\ee
and, for any
\be\label{rs-def}
 \zeta\in(0, 1], \   \ s=\frac{nr_0}{n+(1-\zeta)r_0} \leq r_0,
\ee
there exists $K_{r_0}>0$ independent of $M, L$ such that
 \be\label{S-bdd-imply gradIbdd0}
\sup_{t\in(0, T)} \left\|\nabla I(\cdot,t)\right\|_{L^{r_0}}\leq K_{r_0} \left[1+ M^{q+\frac{qm}{m+1-pm}\left(p-\frac{1}{s}\right)^+} +M^{\frac{qm}{m+1-pm}\left(1-\frac{1}{s}\right)^+}  \right].
 \ee
\end{lemma}

\begin{proof}
In view of the variation-of-constants formula, we have
\be
\nonumber
%\label{I-v-o-c}
0\leq I(\cdot,t) \leq e^{td_I(\Delta-C_0)}I_0 + \int_0^te^{(t-\tau)d_I(\Delta-C_0) } (\beta S^q I^p)(\cdot,\tau)d\tau,
\ee
where $C_0>0$ is such that $\gamma+\mu \geq C_0$ for all $x\in\overline\Omega$ and $t>0$.
Next, we apply  \eqref{S-bdd-imply Ibdd-con}, \eqref{mq-range}  and the  $L^r$-$L^s$-smoothing properties of the Neumann semigroup $\{e^{td_I\Delta}\}_{t\geq0}$  in Lemma \ref{semigroup} along with the maximum principle to derive, for $t\in(0, T)$, that
\be\label{I-infty-est}
\begin{split}
 0\leq I(\cdot,t) &\leq \left\|e^{td_I(\Delta - C_0)} I_0 \right\|_{L^\infty} + \int_0^t \left\| e^{(t-\tau)d_I(\Delta-C_0) }\left(\beta S^q I^p\right)(\cdot,\tau)\right\|_{L^\infty} d\tau\\
   &\leq \|I_0\|_{L^\infty} + C_1\int_0^t\left[1+(t-\tau)^{
   -\frac{n}{2m}}\right]e^{- C_0(t-\tau)}\left\|(S^qI^p)(\cdot,\tau)\right\|_{L^m}d\tau.
   \end{split}
   \ee
We observe from \eqref{L1-id}, \eqref{mq-range} and the definition of $N$ in \eqref{S-bdd-imply Ibdd0}   that, for $\tau \in (0,T)$,
   \be\label{com-est1}
   \begin{split}
&\|\beta (S^qI^p)(\cdot,\tau)\|_{L^m}\leq \beta^0 M^q\|I(\cdot,\tau)\|_{L^{pm}}^p\\
&\leq
\begin{cases}
 \beta^0 M^q \|S_0+I_0\|_{L^1}^{\min\{1,p\}}  N^{(p-1)^+}, & \text{if }  n=1, \\[0.25cm]
 \beta^0 M^q\|S_0+I_0\|_{L^1},  & \text{if }  p< \frac{2}{n},  \   n\geq2, \\[0.25cm]
\beta^0 M^q\|S_0+I_0\|_{L^1}^\frac{1}{m}N^{p-\frac{1}{m}}, & \text{if } \frac{2}{n} \leq  p< 1+ \frac{2}{n}, \  n\geq2.
\end{cases}
\end{split}
\ee
 Thus, we infer from  \eqref{I-infty-est} and \eqref{com-est1}  that
\be\label{I-infty-est2}
L\leq \begin{cases}
C_2\left( 1 + M^q \right),  &\text{if } \left\{ p\leq1, \ n=1\right\} \text{ or } \left\{p< \frac{2}{n}, \ n\geq2\right\}, \\[0.25cm]
C_2\left(1+M^q L^{p-\frac{1}{m}}\right), &\text{if } \left\{1<p< 2,  \ n=1\right\} \text{ or } \left\{p\geq \frac{2}{n}, \ n\geq 2\right\}.
\end{cases}
\ee
In the first case of \eqref{I-infty-est2}, observe that $pm=1$ and \eqref{S-bdd-imply Ibdd0} holds automatically, whereas in  the second scenario,  notice that  the choice of $m$ in \eqref{mq-range} implies $p-\frac{1}{m}<1$, and then we use Young's inequality with epsilon to deduce from  the second case of \eqref{I-infty-est2} that
\begin{align*}
L &\leq  C_2 + C_2 M^q L^{p-\frac{1}{m}}\\
 &\leq C_2 + \frac{1}{2}L+\frac{(m+1-pm)}{m}\left(\frac{2(pm-1)}{m}\right)^\frac{pm-1}{m+1-pm}
 \left(C_2M^q\right)^\frac{m}{m+1-pm},
\end{align*}
which quickly  gives rise to our desired estimate \eqref{S-bdd-imply Ibdd0}.

Next, by the variation-of-constants formula, it holds
\begin{align}
I(\cdot,t) = e^{t d_I \Delta}I_0 + \int_0^t e^{(t-\tau)d_I \Delta}(\beta S^q I^p)(\cdot,\tau)d\tau - \int_0^t e^{(t-\tau)d_I\Delta}((\gamma+\mu)I)(\cdot,\tau)d\tau.
\nonumber
\end{align}
We take the gradient and then use the $L^r$-$L^s$-estimates of the Neumann semigroup $\{e^{td_I\Delta}\}_{t\geq0}$ in Lemma \ref{semigroup} and \eqref{S-bdd-imply Ibdd0} to derive
\begin{equation*}
\begin{split}
\left\|\nabla I(\cdot,t)\right\|_{L^{r_0}}&\leq \left\|\nabla e^{td_I\Delta}I_0\right\|_{L^{r_0}} + \int_0^t \left\|\nabla e^{(t-\tau)d_I\Delta } (\beta S^q I^p)(\cdot,\tau)\right\|_{L^{r_0}}d\tau\\
&\ \ +\int_0^t \left\|\nabla e^{(t-\tau)d_I\Delta }((\gamma+\mu) I)(\cdot,\tau)\right\|_{L^{r_0}} d\tau\\
&\leq \|I_0\|_{W^{1,r_0}} + C_3\int_0^t\left[1+(t-\tau)^{-\frac{1}{2}
-\frac{n}{2}\left(\frac{1}{s}-\frac{1}{r_0}\right)}\right]e^{-\lambda_1 (t-\tau)} \left\| (S^qI^p)(\cdot,t)\right\|_{L^s}d\tau\\
&\ \ + C_3\int_0^t\left[1+(t-\tau)^{-\frac{1}{2}
-\frac{n}{2}\left(\frac{1}{s}-\frac{1}{r_0}\right)}\right]e^{-\lambda_1 (t-\tau)}\|  I(\cdot,t)\|_{L^s}d\tau\\
&\leq C_4\left[1+  M^qL^{ \left(p-\frac{1}{s}\right)^+} + L^{\left(1-\frac{1}{s}\right)^+}\right]
\int_0^\infty\left(1+\tau^{-1+\frac{\zeta}{2}}\right)e^{-\lambda_1 \tau}d\tau,
\end{split}
\end{equation*}
which along with \eqref{S-bdd-imply Ibdd0}  immediately shows  \eqref{S-bdd-imply gradIbdd0}.
\end{proof}

With such preparations, we now close the loop  in the following lemma.

\begin{lemma}
\label{Sbdd-lemma}
Given  $T\in(0, T_m)$, and suppose that \eqref{S-bdd-imply Ibdd-con} and \eqref{mq-range} hold. Then, for $r>\frac{n}{2}$ and $r\geq 1$, one can find constant $K_3>\|S_0\|_{L^\infty}$ independent of $M$ and $L$ such that
\be\label{S-bdd-chismal}
\begin{split}
\sup_{t\in(0, T)}\|S(\cdot,t)\|_{L^\infty}& \leq K_3\Bigg\{1+M^{\frac{qm}{m+1-pm} \left(1-\frac{1}{r}\right)}\\
&\ \ \  \ \  +M|\chi|\left[1+ M^{q+\frac{qm}{m+1-pm} \left(p-\frac{1}{r_0}\right)^+} +M^{\frac{qm}{m+1-pm} \left(1-\frac{1}{r_0}\right)}  \right] \Bigg\}.
\end{split}
 \ee
\end{lemma}

\begin{proof}
We first rewrite the $S$-equation via the  semigroup representation as
\be
\nonumber
%\label{S-v-o-c}
S(t)=e^{td_S\Delta}S_0+\chi\int_0^te^{(t-\tau)d_I\Delta }\nabla\cdot(S\nabla I)d\tau+\int_0^te^{(t-\tau)d_I\Delta } \left(-\beta S^qI^p+\gamma I \right)d\tau.
\ee
Then we use \eqref{S-bdd-imply Ibdd-con}, \eqref{S-bdd-imply Ibdd0} and the well-known  $L^r$-$L^s$-smoothing properties of the Neumann semigroup $\{e^{td_I\Delta}\}_{t\geq0}$ along with the maximum principle to deduce, for $t\in(0, T)$, that
\be\label{S-infty-est}
\begin{split}
S(\cdot,t) &\leq \left\|e^{td_S\Delta}S_0 \right\|_{L^\infty} + |\chi| \int_0^t \left\| e^{(t-\tau)d_S\Delta } \nabla\cdot(S\nabla I)(\cdot,\tau)\right\|_{L^\infty} d\tau\\
 &\ \ \  + \int_0^t \left\| e^{(t-\tau)d_S\Delta }\left(\gamma I\right)(\cdot,\tau)\right\|_{L^\infty} d\tau\\
   &\leq \left\|S_0\right\|_{L^\infty} + C_1|\chi| \int_0^t \left[1+(t-\tau)^{-\frac{1}{2}-\frac{n}{2r_0}}\right]e^{- \lambda_1(t-\tau)}\left\|(S\nabla I)(\cdot,\tau)\right\|_{L^{r_0}}d\tau\\
   & \ \ \ + C_2\int_0^t\left[1+(t-\tau)^{
   -\frac{n}{2r}}\right]e^{- \lambda_1(t-\tau)} \left\| I(\cdot,\tau)\right\|_{L^r}d\tau\\
   &\leq \|S_0\|_{L^\infty} + C_3 |\chi| M \sup_{t\in(0, T)}\left\|\nabla I(\cdot,t)\right\|_{L^{r_0}} + C_4 N^{1-\frac{1}{r}},
   \end{split}
   \ee
   which along with \eqref{S-bdd-imply Ibdd0}, \eqref{rs-def}  and \eqref{S-bdd-imply gradIbdd0}   with $s=r_0$ yields the estimate \eqref{S-bdd-chismal}.
\end{proof}

Now, we are ready to establish our first result on the well-posedness of system \eqref{SIS-mass-C}.

\begin{theorem} \label{glob bdd-semi1} Let $\Omega\subset\mathbb{R}^n \ (n\geq1)$ be a bounded and smooth domain.  Assume that
\be\label{thm-con}
 np+(n-2)^+q<n+\min\left\{n, \  2\right\}.
\ee
 Then, for any initial data $(S_0, I_0)$ fulfilling \eqref{initial data-req}, there exists a positive constant  $\chi_0>0$ (cf. \eqref{chi-max} below) such that whenever $|\chi|<\chi_0$, the corresponding classical solution to the IBVP \eqref{SIS-mass-C} exists globally in time  and is uniformly bounded in the sense there exists a positive constant $C=C(S_0, I_0, \chi, \Omega, p,q, \beta,\gamma,\mu)>0$ such that
\be\label{S-I-bdd-fin1}
\|S(\cdot,t)\|_{L^\infty} + \|I(\cdot,t)\|_{W^{1,r_0}}\leq C, \ \ \ \forall \, t>0.
\ee
\end{theorem}

\begin{proof}
Based on \eqref{mq-range} and  \eqref{thm-con}, we specify $m$ and $r$ as follows:
\be
\nonumber
%\label{mr-take}
  \begin{cases}
 m=\max\left\{1, \ \frac{1}{p}\right\},\  \ r=1,  & \text{if } n=1, \\[0.2cm]
m=\frac{1}{p}, \ \ r=\frac{1}{2}\left[\frac{n}{2}+\min\left\{\frac{q}{(q-1)^+}, \ n\right\}\right], & \text{if  } n\geq2, \   p< \frac{2}{n},  \\[0.2cm]
   m=\frac{n}{2-n\epsilon},\ \ r=\frac{1}{2} \left[ \frac{n}{2} + \min\left\{ \frac{q}{\left[q - \left(1+\frac{2}{n}-p\right) + \epsilon \right]^+}, \ n\right\}\right], & \text{if  } n\geq2, \  \frac{2}{n}\leq p<1+\frac{2}{n},
\end{cases}
 \ee
 where the positive constant $\epsilon$ is given by
 $$
  \epsilon=\frac{1}{2}\min\left\{\frac{2}{n}, \ 1+\frac{2}{n}-p- \left(1-\frac{2}{n}\right) q \right\}>0, \ \ n\geq 2.
 $$
 Then by plain verifications we find that both $m$ and $r$ satisfy the conditions set forth in Lemmas \ref{Ibdd-lemma} and \ref{Sbdd-lemma}, and moreover that
\be\label{qmp-range}
\frac{qm}{m+1-pm} \left(1-\frac{1}{r}\right) < 1.
\ee
For those geometrical constants $K_i$ depending on the initial data and $\Omega$  provided by Lemmas \ref{Ibdd-lemma} and \ref{Sbdd-lemma}, thanks to \eqref{qmp-range}, we first define a positive and finite  constant $\chi_0$ by
\be\label{chi-max}
\begin{split}
\chi_0&=\sup_{z>0}\frac{z-K_3 \left[1+z^{\frac{qm}{m+1-pm} \left(1-\frac{1}{r}\right)}\right]}{K_3 z\left[1+ z^{q+\frac{qm}{m+1-pm} \left(p-\frac{1}{r_0}\right)^+} + z^{\frac{qm}{m+1-pm}\left(1-\frac{1}{r_0}\right)}  \right]}\\
&\ \ =\frac{M-K_3\left[1+M^{\frac{qm}{m+1-pm} \left(1-\frac{1}{r}\right)}\right]} {K_3 M \left[1+ M^{q+\frac{qm}{m+1-pm} \left(p-\frac{1}{r_0}\right)^+} + M^{\frac{qm}{m+1-pm}\left(1-\frac{1}{r_0}\right)}  \right]}\in(0, \infty)
\end{split}
\ee
for some $M>K_3$. This, along with the fact $K_3\geq \|S_0\|_{L^\infty}$, necessarily  implies that
\be\label{M-choose}
\|S_0\|_{L^\infty}<M, \quad \mbox{and}\quad  K_3\left(1+M^{\frac{qm}{m+1-pm} \left(1-\frac{1}{r}\right)}\right)<M.
\ee
Now, for $\chi\in[-\chi_0, \chi_0]$, we let $(S, I)$ be the corresponding  maximally extended solution of the IBVP  \eqref{SIS-mass-C} on $\Omega\times (0, T_m)$ with $T_m\in(0, \infty]$. Thanks to  Lemmas \ref{Ibdd-lemma} and \ref{Sbdd-lemma},  we next employ a standard continuity argument to close the loop-argument and then conclude global existence (i.e., $T_m=\infty$) and boundedness \eqref{S-I-bdd-fin1}. To this end,  we  define
\be\label{set}
\mathcal U:=\Bigl\{(0, T_0)\subset(0, T_m):\  \|S(\cdot, t)\|_{L^\infty}\leq M, \  \forall\, t\in(0, T_0)\Bigr\}.
\ee
Then the continuity of $S$ and \eqref{M-choose} imply  that $\mathcal U$ is nonempty and so $T=\sup \mathcal U$ is well-defined and $T\in(0, \infty]$.  Furthermore, by \eqref{S-bdd-chismal}, \eqref{chi-max} and \eqref{M-choose}, we infer that
\be\label{S-bdd-chismal+}
\begin{split}
\sup_{t\in(0, T)}\|S(\cdot,t)\|_{L^\infty}& \leq K_3\left[1+M^{\frac{qm}{m+1-pm} \left(1-\frac{1}{r}\right)}\right]\\
&\ \ \  \ \  + K_3M|\chi|\left[1+ M^{q+\frac{qm}{m+1-pm} \left(p-\frac{1}{r_0}\right)^+} + M^{\frac{qm}{m+1-pm} \left(1-\frac{1}{r_0}\right)}  \right]\\
&\leq M.
\end{split}
 \ee
 Finally, combining \eqref{S-bdd-chismal+}, \eqref{S-bdd-imply Ibdd0} and  \eqref{S-bdd-imply gradIbdd0} and the extensibility criterion in \eqref{extend} of Lemma \ref{local-in-time}, we must have that $T_m=\infty$ and and then we achieve our claimed  uniform $(L^\infty, W^{1,r_0})$-boundedness of $(S,I)$ on $\Omega\times (0, \infty)$.
\end{proof}

\subsection{Large $\chi$-boundedness and global existence}

In this subsection, we shall present two types of large $\chi$-boundedness of solutions to \eqref{SIS-mass-C} via mainly semi-group method and energy estimate method as illustrated in Subsections 3.2.1 and 3.2.2.

\subsubsection{Large $\chi$-boundedness in any dimensions with  $p,q$ restrictive}

We first use  arguments  in  sprit similar to Subsection 3.1 to  obtain the following large $\chi$-boundedness  in any dimensions while with more restrictive $p$ and $q$.

\begin{theorem}
\label{th+-semi-}
Let $\Omega\subset\mathbb{R}^n \ (n\geq1)$ be a bounded and smooth domain.  Assume that
\be
 q<\frac{1}{n+1}\quad \mbox{and} \quad \ p+(n+1)q < 1+ \min\left\{1,\frac{2}{n}\right\}.
\label{thm-con+}
\ee
 Then, for any initial data $(S_0, I_0)$ fulfilling \eqref{initial data-req}, the corresponding classical solution to the IBVP \eqref{SIS-mass-C} exists globally in time  and is uniformly bounded in the sense  of \eqref{S-I-bdd-fin1}.
\end{theorem}

\begin{proof}
Taking $r_0=\infty$ and $s=\frac{n}{1-\zeta}$ in \eqref{rs-def} with $\zeta\in(0,1]$, we get  from \eqref{S-bdd-imply gradIbdd0} that
 \be
\label{S-bdd-imply gradIbdd-infty}
\sup_{t\in(0, T)}\|\nabla I(\cdot,t)\|_{L^\infty} \leq K_\zeta \left[1+ M^{q+\frac{qm}{m+1-pm}\left(p-\frac{1}{n}+\frac{\zeta}{n}\right)^+} + M^{\frac{qm}{m+1-pm}\left(1-\frac{1}{n}+\frac{\zeta}{n}\right)^+}  \right].
 \ee
With this, for any $r>\frac{n}{2}$ and $r\geq 1$, we re-bound \eqref{S-infty-est} as follows:
 \be\label{S-infty-est+}
\begin{split}
S(\cdot,t) &\leq \left\|e^{td_S\Delta}S_0\right\|_{L^\infty} + |\chi|\int_0^t \left\| e^{(t-\tau)d_S\Delta } \nabla\cdot(S\nabla I)(\cdot,\tau)\right\|_{L^\infty} d\tau\\
 &\ \ \  + \int_0^t \left\| e^{(t-\tau)d_S\Delta }\left(\gamma I\right)(\cdot,\tau)\right\|_{L^\infty} d\tau\\
   &\leq \|S_0\|_{L^\infty} + C_1|\chi|\int_0^t\left[1+(t-\tau)^{-\frac{1}{2}-\frac{n}{4r}}\right]e^{- \lambda_1(t-\tau)}\left\|(S\nabla I)(\cdot,\tau)\right\|_{L^{2r}}d\tau\\
   & \ \ \  + C_2\int_0^t\left[1+(t-\tau)^{ -\frac{n}{2r}}\right]e^{- \lambda_1(t-\tau)}\| I(\cdot,\tau)\|_{L^r}d\tau\\
   &\leq \|S_0\|_{L^\infty}+C_3|\chi|M^{1-\frac{1}{2r}} \sup_{t\in(0, T)}\|\nabla I(\cdot,t)\|_{L^\infty}+C_4 N^{1-\frac{1}{r}},
   \end{split}
   \ee
 which along with \eqref{S-bdd-imply gradIbdd-infty} and \eqref{S-bdd-imply Ibdd0} allows us to infer
 \be\label{S-infty-est+}
\begin{split}
 \sup_{t\in(0, T)}\|S(\cdot,t)\|_{L^\infty} & \leq C_5 \Bigg[1 + |\chi| M^{1-\frac{1}{2r}} + |\chi| M^{q+1-\frac{1}{2r}+\frac{qm}{m+1-pm}
   \left(p-\frac{1}{n}+\frac{\zeta}{n}\right)^+}\\
   & \ \ \ +|\chi| M^{1-\frac{1}{2r}+\frac{qm}{m+1-pm}
   \left(1-\frac{1}{n}+\frac{\zeta}{n}\right)^+} + M^{\frac{qm}{m+1-pm}\left(1-\frac{1}{r}\right)}\Bigg].
   \end{split}
   \ee
 Now, we observe  that \eqref{thm-con+} implies that
\begin{equation*}
\left\{ \begin{array}{lll}
q+1-\frac{1}{n}+\frac{q\frac{n}{2}}{\frac{n}{2}+1-p\frac{n}{2}} \left(p-\frac{1}{n}\right)^+<1,\\[0.3cm]
1-\frac{1}{n}+\frac{q\frac{n}{2}}{\frac{n}{2}+1-p\frac{n}{2}} \left(1-\frac{1}{n}\right)^+<1, \\[0.3cm]
\frac{q\frac{n}{2}}{\frac{n}{2}+1-p\frac{n}{2}}\left(1-\frac{2}{n}\right)^+<1,
\end{array}\right.
\end{equation*}
which upon continuity allows us to fix $0<\epsilon<\min\left\{\frac{2}{n}, \ 1+\frac{2}{n}-p\right\}$ (cf. \eqref{mq-range})  such that
\be\label{epsilon-choise}
\begin{cases}
q+1-\frac{1}{n+\epsilon}+\frac{q\frac{n}{2-n\epsilon}}{\frac{n}{2-n\epsilon}+1
-p\frac{n}{2-n\epsilon}}\left(p-\frac{1-\epsilon}{n}\right)^+<1,\\[0.25cm]
1-\frac{1}{n+\epsilon}+\frac{q\frac{n}{2-n\epsilon}}{\frac{n}{2-n\epsilon}
+1-p\frac{n}{2-n\epsilon}} \left(1-\frac{1-\epsilon}{n}\right)^+<1, \\[0.25cm]
\frac{q\frac{n}{2-n\epsilon}}{\frac{n}{2-n\epsilon}+1-p\frac{n}{2-n\epsilon}} \left(1-\frac{2}{n+\epsilon}\right)^+<1.
\end{cases}
\ee
Based on \eqref{thm-con+} and \eqref{S-infty-est+}, we first choose $m$, $r$  and $\zeta$ as follows:
   \be\label{mr-take+}
  \begin{cases}
 m=\frac{1}{p},\  \ r=1,  \  \  \zeta=\frac{1}{2}\min\left\{1-p, \ \frac{1}{2q}\right\},& \text{if } n=1,\  p<1, \\[0.2cm]
 m=1,\  \ r=1,  \ \   \zeta=\frac{2-p-2q}{3q}, & \text{if } n=1, \  p\geq 1, \\[0.2cm]
m=\frac{1}{p},   \ r=\frac{1}{2}\left\{\frac{n}{2}+   \frac{n}{2q[(np-1)^++n]} \right\}, & \text{if  } n\geq2, \  p<\frac{2}{n},  \\[0.2cm]
   m=\frac{n}{2-n\zeta},\ r=\frac{n+\zeta}{2}, \ \zeta = \epsilon, & \text{if  } n\geq2,\  p\geq  \frac{2}{n}.
\end{cases}
 \ee
 In the third case of \eqref{mr-take+}, $\zeta>0$ is chosen so that
\begin{equation*}
\left\{ \begin{array}{ll}
p-\frac{1}{n}+\frac{\zeta}{n}<0,&\text{if }p< \frac{1}{n},\\[0.25cm]
\left(np-1+n+\zeta\right) \left[\frac{q}{2}+\frac{1}{2(np-1+n)} \right] < 1,&\text{if }p\geq \frac{1}{n},
\end{array}\right.
\end{equation*}
which is feasible since
\begin{align*}
\left(np-1+n\right) \left[\frac{q}{2}+\frac{1}{2(np-1+n)} \right] = \frac{1}{2}+\frac{q}{2}(np-1+n) < \frac{1}{2} + \frac{q}{2}(n+1) <1,
\end{align*}
for $ p \geq \frac{1}{n}$ and $(p,q)$ fulfilling \eqref{thm-con+}.
 With such chosen $m$, $r$ and $\zeta$, we then compute from \eqref{S-infty-est+}, \eqref{epsilon-choise} and \eqref{mr-take+} that
\be\label{coeff<1}
\begin{cases}
 q+1-\frac{1}{2r}+\frac{qm}{m+1-pm} \left(p-\frac{1}{n}+\frac{\zeta}{n}\right)^+<1, \\[0.2cm]
 1-\frac{1}{2r}+\frac{qm}{m+1-pm} \left(1-\frac{1}{n}+\frac{\zeta}{n}\right)^+<1, \\[0.2cm]
 \frac{qm}{m+1-pm}\left(1-\frac{1}{r}\right) <1.
\end{cases}
   \ee
In light of \eqref{coeff<1}, one can easily use Young's inequality with epsilon to \eqref{S-infty-est+} to derive the existence of $M=M(S_0, I_0, \beta, \gamma,\mu, \Omega, \chi,, p,q)>\|S_0\|_{L^\infty}$ such  that
\begin{equation*}
\begin{split}
 \sup_{t\in(0, T)}\|S(\cdot,t)\|_{L^\infty}&\leq C_5\Bigr[1+|\chi| M^{1-\frac{1}{2r}}+|\chi| M^{q+1-\frac{1}{2r}+\frac{qm}{m+1-pm}
   \left(p-\frac{1}{n}+\frac{\zeta}{n}\right)^+}\\
   & \ \ \quad +|\chi| M^{1-\frac{1}{2r}+\frac{qm}{m+1-pm}
   \left(1-\frac{1}{n}+\frac{\zeta}{n}\right)^+}+M^{\frac{qm}{m+1-pm}\left(1-\frac{1}{r}\right)}\Bigr] \\
  &\leq M.
   \end{split}
\end{equation*}
 With this, we see the set defined by \eqref{set} is a nonempty and both open and closed set, and thus, using similar  continuity arguments as done in  Theorem \ref{glob bdd-semi1}, we readily obtain that $T_m=\infty$ and that  $(S, I)$ is uniform-in-time bounded according to \eqref{S-I-bdd-fin1}.
\end{proof}

\subsubsection{Large $\chi$-boundedness in lower dimentions}
In this subsection,  we shall employ coupled subtle energy estimates  to improve  the  boundedness range in \eqref{thm-con+} for $n\leq2$.

To move from $L^1$-boundness  obtained in \eqref{L1-id} to  higher order regularity for the local solution $(S, I)$ of \eqref{SIS-mass-C}, we begin to study the time evolutions of $\|(S+1)\ln (S+1)\|_{L^1}$ and $\|\nabla I\|_{L^2}^2$; upon integration by parts on  \eqref{SIS-mass-C}, we find  the following energy identities.
\begin{lemma}The local-in-time solution of \eqref{SIS-mass-C} satisfies
\be\label{SlnS-id}
\begin{split}
&\frac{d}{dt}\int_\Omega (S+1)\ln (S+1)
+4d_S\int_\Omega \left|\nabla (S+1)^\frac{1}{2}\right|^2+\beta \int_\Omega S^ q I^ p  [\ln (S+1)+1]\\
&\ \ =\chi\int_\Omega [S-\ln (S+1)]\Delta I+\gamma  \int_\Omega I [\ln (S+1)+1]
\end{split}
\ee
and
\be\label{gradI-id}
\frac{1}{2}\frac{d}{dt}\int_\Omega |\nabla I|^2+(\gamma+\mu)\int_\Omega |\nabla I|^2+d_I\int_\Omega |\Delta I|^2=-\beta \int_\Omega S^ q I^ p  \Delta I.
\ee
\end{lemma}

To bound the bad terms on the right-hand sides of \eqref{SlnS-id} and \eqref{gradI-id} in terms of the dissipation terms on their left-hand sides, based on the $L^1$-boundedness of $S+I$ in \eqref{L1-id}, we employ  the G-N inequality (cf. Lemma \ref{GNI}) to deduce the following estimates.

\begin{lemma}
\label{GN-by-L1}
{\rm (i)} Let $\Omega\subset\mathbb{R}^1$ be a finite interval. Then, for $r\in(1,\infty]$ and $\epsilon>0$, it holds
\be\label{S-by-Gn1}
\begin{cases}
\left\|S+1\right\|_{L^r}^\frac{2r}{r-1}\leq C \left\|\nabla (S+1)^\frac{1}{2}\right\|_{L^2}^2+C,\\[0.25cm]
\forall s<\frac{2r}{r-1}, \ \ \left\|S+1\right\|_{L^r}^s\leq \epsilon \left\|\nabla (S+1)^\frac{1}{2}\right\|_{L^2}^2+C_\epsilon, \\[0.25cm]
 \left\|(S+1)\ln (S+1)\right\|_{L^1}\leq \epsilon \left\|\nabla (S+1)^\frac{1}{2}\right\|_{L^2}^2+C_\epsilon
\end{cases}
\ee
and
\be
\label{I-by-Gn1}
\begin{cases}
\|I\|_{L^r}^\frac{5r}{r-1}\leq C\|\Delta I\|_{L^2}^2+C\|\nabla I\|_{L^2}^2+C,\\[0.25cm]
\forall s<\frac{5r}{r-1}, \ \ \|I\|_{L^r}^s\leq \epsilon \|\Delta I\|_{L^2}^2+\epsilon \|\nabla I\|_{L^2}^2+C_\epsilon.
\end{cases}
\ee
{\rm (ii)} Let $\Omega\subset\mathbb{R}^2$ be bounded and smooth. Then, for $r\in(1,\infty)$  and $\epsilon>0$, it holds
\be\label{S-by-Gn}
\begin{cases}
\|S+1\|_{L^r}^\frac{r}{r-1}\leq C\left\|\nabla (S+1)^\frac{1}{2}\right\|_{L^2}^2+C,\\[0.25cm]
\forall s<\frac{r}{r-1}, \ \ \|S+1\|_{L^r}^s\leq \epsilon \left\|\nabla (S+1)^\frac{1}{2}\right\|_{L^2}^2+C_\epsilon, \\[0.25cm]
 \|(S+1)\ln (S+1)\|_{L^1}\leq \epsilon \left\|\nabla (S+1)^\frac{1}{2}\right\|_{L^2}^2+C_\epsilon
\end{cases}
\ee
and, for $r\in(1,\infty]$ and $\epsilon>0$, it holds
\be\label{I-by-Gn}
\begin{cases}
\|I\|_{L^r}^\frac{3r}{r-1}\leq C\|\Delta I\|_{L^2}^2+C\|\nabla I\|_{L^2}^2+C,\\[0.25cm]
\forall s<\frac{3r}{r-1}, \ \ \|I\|_{L^r}^s\leq \epsilon \|\Delta I\|_{L^2}^2+\epsilon \|\nabla I\|_{L^2}^2+C_\epsilon.
\end{cases}
\ee
\end{lemma}

\begin{proof}
(i) By the $L^1$-bound of $S$ in \eqref{L1-id}, we apply the 1-D G-N inequality as stated in Lemma \ref{GNI} to infer, for $r\in (1, \infty]$,
\begin{align*}
 \|S+1\|_{L^r}^s&= \left\|(S+1)^\frac{1}{2}\right\|_{L^{2r}}^{2s}\\
 &\leq C \left\|\nabla (S+1)^\frac{1}{2}\right\|_{L^2}^\frac{(r-1)s}{r} \left\|(S+1)^\frac{1}{2}\right\|_{L^2}^\frac{(r+1)s}{r}
  +C \left\|(S+1)^\frac{1}{2}\right\|_{L^2}^{2s}\\
  &\leq C\left\|\nabla (S+1)^\frac{1}{2}\right\|_{L^2}^\frac{(r-1)s}{r}+C\\
  &\begin{cases} = C\left\|\nabla (S+1)^\frac{1}{2}\right\|_{L^2}^2+C,  &\text{ if  } s=\frac{2r}{r-1}, \\[0.25cm]
  \leq \epsilon \left\|\nabla (S+1)^\frac{1}{2}\right\|_{L^2}^2+C_\epsilon, &\text{ if  } s<\frac{2r}{r-1}.
  \end{cases}  \end{align*}
This,  together with the simple algebraic fact that
\be\label{slns-in}
\|(S+1)\ln(S+1)\|_{L^1}\leq \|S+1\|_{L^\frac{3}{2}}^\frac{3}{2},
\ee
shows the desired estimates stated in  \eqref{S-by-Gn1}.

As for the bounds for $I$ in \eqref{I-by-Gn1}, in view of the  $L^1$-boundedness of $I$ guaranteed by \eqref{L1-id}, we use  the 1-D G-N inequality to derive, for $r\in (1, \infty]$, that
\be\label{Ilp-bdd-gn1} \begin{split}
\|I\|_{L^r}^s&\leq C\left\|D^2I\right\|_{L^2}^\frac{2(r-1)s}{5r} \left\|I\right\|_{L^1}^\frac{(3r+2)s}{5r}+C\|I\|_{L^1}^s\\
&\leq C \left\|D^2I\right\|_{L^2}^\frac{2(r-1)s}{5r}+C\\
&\begin{cases} = C\left\|D^2I\right\|_{L^2}^2+C,  &\text{ if  } s=\frac{5r}{r-1}, \\[0.25cm]
  \leq \epsilon \left\|D^2I\right\|_{L^2}^2+C_\epsilon, &\text{ if  } s<\frac{5r}{r-1}.
  \end{cases}
\end{split}
\ee
In a similar manner, we also have, for bounded smooth  $\Omega\subset\mathbb{R}^n(n\geq1)$ and any $\epsilon>0$,
\be\label{Il2-by-gradI-l2}
\|I\|_{L^2}^2\leq \epsilon \|\nabla I\|_{L^2}^2+C_\epsilon.
\ee
On the other hand, for bounded smooth $\Omega\subset\mathbb{R}^n(n\geq1)$, the $H^2$-regularity shows
\be\label{H2-regularity}
\|D^2I\|_{L^2}^2\leq C\|\Delta I\|_{L^2}^2+C\|I\|_{L^2}^2.
\ee
Joining this with \eqref{Ilp-bdd-gn1} and \eqref{Il2-by-gradI-l2}, we end up with the estimates claimed in  \eqref{I-by-Gn1}.

(ii) By the $L^1$-bound of $S$ in \eqref{L1-id}, we apply the 2D G-N inequality as stated in Lemma \ref{GNI} to obatin, for $r\in (1, \infty)$,
\begin{align*}
 \|S+1\|_{L^r}^s&= \left\|(S+1)^\frac{1}{2}\right\|_{L^{2r}}^{2s}\\
 &\leq C \left\|\nabla (S+1)^\frac{1}{2}\right\|_{L^2}^\frac{2(r-1)s}{r} \left\|(S+1)^\frac{1}{2}\right\|_{L^2}^\frac{2s}{r}
  +C \left\|(S+1)^\frac{1}{2}\right\|_{L^2}^{2s}\\
  &\leq C \left\|\nabla (S+1)^\frac{1}{2}\right\|_{L^2}^\frac{2(r-1)s}{r}+C\\
  &\begin{cases} = C \left\|\nabla (S+1)^\frac{1}{2}\right\|_{L^2}^2+C,  &\text{ if  } s=\frac{r}{r-1}, \\[0.25cm]
  \leq \epsilon \left\|\nabla (S+1)^\frac{1}{2}\right\|_{L^2}^2+C_\epsilon, &\text{ if  } s<\frac{r}{r-1}.
  \end{cases}  \end{align*}
This together with the simple fact \eqref{slns-in} entails  the desired estimates stated in  \eqref{S-by-Gn}.

To achieve the bounds for $I$ in \eqref{I-by-Gn},  we use the $L^1$-boundedness of $I$ ensured by  \eqref{L1-id} and the  2D G-N inequality to deduce, for $r\in (1, \infty]$, that
  \begin{align*}
\|I\|_{L^r}^s&\leq C\left\|D^2I\right\|_{L^2}^\frac{2(r-1)s}{3r}\|I\|_{L^1}^\frac{(r+2)s}{3r}+C\|I\|_{L^1}^s\\
&\leq C\left\|D^2I\right\|_{L^2}^\frac{2(r-1)s}{3r}+C\\
&\begin{cases} = C\left\|D^2I\right\|_{L^2}^2+C,  &\text{ if  } s=\frac{3r}{r-1}, \\[0.25cm]
  \leq \epsilon \left\|D^2I\right\|_{L^2}^2+C_\epsilon, &\text{ if  } s<\frac{3r}{r-1}.
  \end{cases}
\end{align*}
This along with \eqref{Il2-by-gradI-l2} and \eqref{H2-regularity} yields easily \eqref{I-by-Gn}.
\end{proof}
Thanks to  the $L^1$-boundedness of $S+I$ in \eqref{L1-id} and Lemma  \ref{GN-by-L1}, we now deduce higher order regularity for $(S, I)$ under certain restrictions of $ q$ and $ p $.
\begin{lemma} \label{SlnS-gradI-bdd} Let  $ q$ and $ p $ satisfy
\be\label{kappa-lambda-con}
\begin{cases}
 10 q+4 p <15, \ \ \ \  q+ p <3,  & \text{ if } n=1, \\[0.25cm]
  3 q+ p <3, \ \ \ \  q+ p <2,  & \text{ if } n=2.
  \end{cases}
\ee
Then the local-in-time solution of \eqref{SIS-mass-C} satisfies
\be\label{SlnS-grad I-est}
 \|(S+1)\ln (S+1)\|_{L^1}+\|\nabla I\|_{L^2}\leq C,\quad \quad \forall \, t\in (0,T_m).
\ee
%and
%\be\label{SlnS-grad I-st-est}
%\int_0^t\left(\int_\Omega \left|\nabla (S+1)^\frac{1}{2}\right|^2+\int_\Omega |\nabla I|^2+\int_\Omega |\Delta I|^2\right)\leq C(1+t).
%\ee
Moreover, there exists $C>0$ such that
\be\label{I-linfty-un-bdd}
\|I(\cdot,t)\|_{L^\infty}\leq C, \quad\quad \forall \, t\in (0, T_m).
\ee
\end{lemma}
\begin{proof}
Due to  the $L^1$-boundedness of $S+I$, Lemma  \ref{GN-by-L1} and \eqref{Il2-by-gradI-l2}, we  bound the bad terms on the right-hand sides of \eqref{SlnS-id} and \eqref{gradI-id} in terms of the dissipation terms on their left-hand sides in the following manners:
\be\label{trouble-1}
\begin{split}
\chi\int_\Omega [S-\ln (S+1)]\Delta I&\leq \epsilon \int_\Omega [S-\ln (S+1)]^2+\frac{\chi^2}{4\epsilon}\int_\Omega |\Delta I|^2\\
&\leq \epsilon \int_\Omega (S+1)^2+\frac{\chi^2}{4\epsilon}\int_\Omega |\Delta I|^2\\
&\leq C \epsilon \int_\Omega \left|\nabla (S+1)^\frac{1}{2} \right|^2 + C\epsilon + \frac{\chi^2}{4\epsilon}\int_\Omega |\Delta I|^2,
\end{split} \ee
\be\label{trouble-2}
\begin{split}
\gamma  \int_\Omega I [\ln (S+1)+1]&\leq \gamma \int_\Omega\ln^2(S+1)+\gamma \int_\Omega  I^2+\gamma  \int_\Omega I\\
&\leq\gamma  \int_\Omega (S+I)+ \gamma \int_\Omega  I^2\\
&\leq\gamma  \int_\Omega (S_0+I_0)+ \gamma \epsilon \int_\Omega  |\nabla I|^2+\gamma  C_\epsilon
\end{split}
\ee
and
\be\label{trouble-3}
-\beta \int_\Omega S^ q I^ p  \Delta I\leq \frac{d_I}{4}  \int_\Omega|\Delta I|^2+\frac{\beta^2}{d_I}  \int_\Omega S^{2 q} I^{2 p }.
\ee
In the sequel, we apply H\"{o}lder inequality and Young's inequality with epsilon  to estimate the second integral on the right-hand of \eqref{trouble-3} as follows: In the simple case of
\be\label{kappa-lambda-cond0}
  q+ p \leq \frac{1}{2},
 \ee
 we have from H\"{o}lder inequality that
 \be\label{kappa-lambda-cond1-est0}
\int_\Omega S^{2 q} I^{2 p }\leq \|S\|_{L^1}^{2 q}\|I\|_{L^\frac{2 p }{1-2 q}}^{2 p }\leq \|S\|_{L^1}^{2 q}\|I\|_{L^1}^{2 p }|\Omega|^{(1-2 q-2 p )}\leq \|S_0+I_0\|_{L^1}|\Omega|^{(1-2 q-2 p )}.
\ee
In the opposite scenario of \eqref{kappa-lambda-cond0}, we break our argument into two cases:

\textbf{Case of  $n=1$}. Because of  \eqref{kappa-lambda-con},  we can take $r\in(1,\infty)$ such that
$$
\max\left\{1, \ \ 2 p , \ \ \frac{6 p }{15-10 q-4 p }\right\}< r<\min\left\{\frac{2 p }{(2 p -5)^+},\ \  \frac{2 p }{(1-2 q)^+}\right\},
 $$
 which imply
$$
r>1, \    r>2 p ,  \   2 p <\frac{5r}{r-1}, \
\frac{2 q r}{r-2 p }=:s>1, \  2q<\frac{10 q r}{5r-2 p (r-1)}<\frac{2s}{s-1};
$$
with such chosen $r$,  we utilize  H\"{o}lder and Young's inequality with epsilon, \eqref{S-by-Gn1} and \eqref{I-by-Gn1} to estimate, for any $\epsilon, \eta>0$, that
\be\label{kappa-lambda-cond1-est6-1d}
\begin{split}
  \int_\Omega S^{2 q} I^{2 p }&\leq \left\|S+1\right\|_{L^\frac{2 q r}{r-2 p }}^{2 q}\|I\|_{L^r}^{2 p }\\
  &\leq \epsilon \|I\|_{L^r}^\frac{5r}{r-1}+ C_\epsilon\left\|S+1\right\|_{L^\frac{2 q r}{r-2 p }}^\frac{10 q r}{5r-2 p (r-1)}\\
  &\leq C\epsilon\|\Delta I\|_{L^2}^2+C\epsilon \|\nabla I\|_{L^2}^2+C\epsilon +\eta  \left\|\nabla (S+1)^\frac{1}{2}\right\|_{L^2}^2+C_{\epsilon, \eta}.
\end{split}
\ee

\textbf{Case of $n=2$}.  In this case,  under   \eqref{kappa-lambda-con}, we will show that \eqref{kappa-lambda-cond1-est6-1d} type estimate also holds. Indeed,   because of \eqref{kappa-lambda-con}, we can take $r\in(1,\infty)$ such that
$$
 \max\left\{1, \  \ 2p , \ \ \frac{2 p }{3-3 q- p }\right\}< r<\min\left\{\frac{2p }{(2 p -3)^+},\ \  \frac{2 p }{(1-2 q)^+}\right\},
 $$
 which imply
$$
r>1, \ \  r>2 p ,  \ \  2 p <\frac{3r}{r-1}, \ \
\frac{2 q r}{r-2 p }=:s>1, \ \ 2q<\frac{6 q r}{3r-2 p (r-1)}<\frac{s}{s-1};
$$
with such chosen $r$,  we utilize  H\"{o}lder and Young's inequality with epsilon, \eqref{S-by-Gn} and \eqref{I-by-Gn} to infer, for any $\epsilon, \eta>0$, that
\be\label{kappa-lambda-cond1-est6}
\begin{split}
  \int_\Omega S^{2 q} I^{2 p }&\leq \|S+1\|_{L^\frac{2 q r}{r-2 p }}^{2 q}\|I\|_{L^r}^{2 p }\\
  &\leq \epsilon\|I\|_{L^r}^\frac{3r}{r-1}+ C_\epsilon\|S+1\|_{L^\frac{2 q r}{r-2 p }}^\frac{6 q r}{3r-2 p (r-1)}\\
  &\leq C\epsilon\|\Delta I\|_{L^2}^2+C\epsilon \|\nabla I\|_{L^2}^2+C\epsilon +\eta  \left\|\nabla (S+1)^\frac{1}{2}\right\|_{L^2}^2+C_{\epsilon, \eta}.
\end{split}
\ee

In summary, we substitute \eqref{trouble-1}, \eqref{trouble-2}, \eqref{trouble-3}, \eqref{kappa-lambda-cond1-est6-1d}, \eqref{kappa-lambda-cond1-est0} and \eqref{kappa-lambda-cond1-est6} into \eqref{SlnS-id} and \eqref{gradI-id} to conclude, for any $\epsilon, \eta>0$, there exist constants  $M, M_\epsilon, L, L_\eta>0$ such that
\be\label{SlnS-est}
\frac{d}{dt}\int_\Omega (S+1)\ln (S+1)
+3d_S\int_\Omega \left|\nabla (S+1)^\frac{1}{2}\right|^2 \leq \gamma  \epsilon\int_\Omega |\nabla I|^2+M\chi^2\int_\Omega |\Delta I|^2+M_\epsilon,
\ee
and
\be\label{gradI-est}
\frac{d}{dt}\int_\Omega |\nabla I|^2+(\gamma+\mu)\int_\Omega |\nabla I|^2+d_I\int_\Omega |\Delta I|^2\leq  L\eta\int_\Omega \left|\nabla (S+1)^\frac{1}{2}\right|^2+L_\eta.
\ee
Multiplying \eqref{SlnS-est} by $d_I$ and \eqref{gradI-est} by $2M\chi^2$ and then choosing sufficiently small $\epsilon$ and $\eta$ and then using Lemma \ref{GN-by-L1}, we finally end up with, for some $K>0$,
\be\nonumber
%\label{SlnS-gradI-odi}
\begin{split}
&\frac{d}{dt}\int_\Omega\left[d_I(S+1)\ln (S+1)+2M\chi^2|\nabla I|^2\right]+d_I\int_\Omega (S+1)\ln (S+1) \\
&\ \ +(\gamma+\mu) M\chi^2\int_\Omega |\nabla I|^2+d_Sd_I\int_\Omega \left|\nabla (S+1)^\frac{1}{2}\right|^2 + d_I M\chi^2\int_\Omega |\Delta I|^2\leq K.
\end{split}
\ee
Solving this standard Grownall inequality, we directly obtain the uniform estimate \eqref{SlnS-grad I-est}.
%, and then integrating \eqref{SlnS-gradI-odi} from $0$ to $t$, we obtain the space-time estimate \eqref{SlnS-grad I-st-est}.

In the case of $n=1$,  by \eqref{SlnS-grad I-est} and  \eqref{Il2-by-gradI-l2}, we know that  the $W^{1,2}$-norm of $I$ is uniformly bounded, and thus \eqref{I-linfty-un-bdd} follows from the continuous embedding of  $W^{1,2}(\Omega)$  into  $L^\infty(\Omega)$.

Next, in the case of $n=2$,   the $W^{1,2}$-boundedness of $I$  implied by  \eqref{SlnS-grad I-est} and   \eqref{Il2-by-gradI-l2} yield first  the uniform boundedness of $\|I\|_{L^q}$ for any $q\in (0,\infty)$:
\be\label{I-lq-un-bdd}
\|I(\cdot,t)\|_{L^q}\leq C_q, \ \ \ \forall\, t\in (0, T_m).
\ee
 To show the $L^\infty$-boundedness of $I$, we use semi-group representation to rewrite $I$ as
\be\label{I-rewriting}
   I=e^{t(d_I\Delta-(\gamma+\mu))}I_0+\beta\int_0^te^{(t-s)(d_I\Delta-(\gamma+\mu))}S^ q I^ p  ds.
 \ee
 Based on the  $L^r$-$L^s$ type estimates in Lemma \ref{semigroup},  the $L^1$-boundedness of $S$, \eqref{I-lq-un-bdd} and the  fact that $ q<1$ ensured by \eqref{kappa-lambda-con}, we apply  H\"{o}lder interpolation inequality  to estimate
   \begin{align*}
   \|I(\cdot,t)\|_{L^\infty}&\leq \left\|e^{t(d_I\Delta-(\gamma+\mu))}I_0\right\|_{L^\infty} + \beta \int_0^t \left\| e^{(t-s)(d_I\Delta-(\gamma+\mu))}S^ q I^ p \right\|_{L^\infty} ds\\
   &\leq C\|I_0\|_{L^\infty}+C\int_0^t\left[1+(t-s)^{
   -\frac{2 q}{ q+1}}\right]e^{-(\gamma+\mu)(t-s)}\left\|S^ q I^ p \right\|_{L^\frac{ q+1}{2 q}}ds\\
   &\leq C+C\int_0^t\left[1+(t-s)^{
   -\frac{2 q}{ q+1}}\right]e^{-(\gamma+\mu)(t-s)}\|S^ q\|_{L^\frac{1}{ q}} \|I^ p \|_{L^\frac{ q+1}{(1- q) q}}ds\\
   &\leq C+C\int_0^t\left[1+(t-s)^{
   -\frac{2 q}{ q+1}}\right]e^{-(\gamma+\mu)(t-s)}ds\\
   &\leq C+C\int_0^\infty\left(1+\tau^{
   -\frac{2 q}{ q+1}}\right)e^{-(\gamma+\mu)\tau}d\tau\leq C,
   \end{align*}
   which precisely gives rise to \eqref{I-linfty-un-bdd}.
\end{proof}

Based on the basic but crucial  starting boundedness provided in Lemma \ref{SlnS-gradI-bdd},  we use test-procedure to obtain the $L^2$-bound for $S$ and $|\nabla I|^4$ by estimating the combined  time evolution of $\|S\|_{L^2}^2$ and $\|\nabla I\|_{L^4}^4$.

\begin{lemma}\label{L2-bdd}
Let \eqref{kappa-lambda-con} hold.  Then the local-in-time solution of \eqref{SIS-mass-C} satisfies
\be\label{S-l2-grad I-l4-est}
 \|S(\cdot,t)\|_{L^2}+\|\nabla I(\cdot,t)\|_{L^4}\leq C, \quad \quad \forall\, t\in (0, T_m).
\ee
%and
%\be\label{S-l2-grad I-l4-st-est}
%\int_0^t \int_\Omega \left(|\nabla S|^2+\left|\nabla I\right|^4 +  \left|\nabla\left|\nabla I\right|^2 \right|^2 +  S^{1+ q} I^ p \right)\leq C(1+t),\  \forall t\in (0, T_m).
%\ee
\end{lemma}

\begin{proof}
We test the $S$-equation in \eqref{SIS-mass-C} by $S$  and then integrate by parts to derive  that
\begin{align*}
&\frac{1}{2}\frac{d}{dt}\int_\Omega S^2+ d_S\int_\Omega|\nabla S|^2+\beta \int_\Omega S^{1+ q} I^ p \\ & = -\chi \int_\Omega S\nabla S \cdot \nabla I +\gamma  \int_\Omega S I \\
&\leq \frac{d_S}{4}\int_{\Omega}|\nabla S|^2 + \frac{\chi^2}{d_S} \int_{\Omega}S^2|\nabla I|^2 + \gamma \|I\|_{L^\infty}\int_{\Omega} S,
\end{align*}
which, together with \eqref{L1-id} and \eqref{I-linfty-un-bdd}, gives
\begin{align}
\label{S-l2-id}
\frac{1}{2}\frac{d}{dt}\int_\Omega S^2+ \frac{3d_S}{4}\int_\Omega|\nabla S|^2+\beta \int_\Omega S^{1+ q} I^ p \leq \frac{\chi^2}{d_S} \int_{\Omega}S^2|\nabla I|^2 +C_1.
\end{align}
Next, taking  gradient of the $I$-equation  and then multiplying it by $\nabla I|\nabla I|^2$ and integrating  by parts, we conclude that
\be\label{gradI-l4-id}\begin{split}
&\frac{1}{2}\frac{d}{dt} \int_\Omega |\nabla I|^4 + d_I \int_\Omega \left|\nabla \left|\nabla I\right|^2\right|^2 + 2 d_I\int_\Omega  \left|\nabla I\right|^{2} \left|D^2I \right|^2 + 2(\gamma+\mu)\int_\Omega \left|\nabla I\right|^{4}\\
&=\int_{\partial\Omega}  |\nabla I|^{2}\frac{\partial}{\partial \nu} \left|\nabla I\right|^2 - 2\beta\int_\Omega  S^ q I^ p \Delta  I|\nabla I|^{2} - 2\beta\int_\Omega  S^ q I^ p \nabla  I \cdot \nabla |\nabla I|^{2}.  \end{split}
\ee
In the above, we have used the identity
 $$2 \nabla I \cdot \nabla \Delta I = \Delta |\nabla I|^2 - 2 \left|D^2 I\right|^2,$$
where $\left|D^2 I\right|^2 = \sum_{i,j=1}^n\left|\frac{\partial^2 I}{\partial x_i \partial x_j}\right|^2$.
For the boundary integral in \eqref{gradI-l4-id}, we can handle it (cf.  \cite{ISY14, TWZAMP, Xiang16-ppt, Xiangpre2}) as follows
\be\label{com-est-4}
\begin{split}
\int_{\partial\Omega}  \left|\nabla I\right|^{2} \frac{\partial}{\partial \nu} \left|\nabla I\right|^2&\leq \epsilon \int_\Omega \left|\nabla \left|\nabla I\right|^2\right|^2 + C_\epsilon \left(\int_{\Omega} |\nabla I|^2\right)^2\\
&\leq \epsilon \int_\Omega \left|\nabla \left|\nabla I\right|^2\right|^2+C_\epsilon,\quad \quad  \forall \epsilon>0.
\end{split}
\ee
For the remaining terms in \eqref{gradI-l4-id}, we apply H\"{o}lder and Young's inequality to deduce that
\be\label{com-est-1}\begin{split}
& -2\beta\int_\Omega  S^ q I^ p \Delta  I|\nabla I|^{2}-2\beta\int_\Omega  S^ q I^ p \nabla  I\cdot \nabla \left|\nabla I\right|^{2}  \\
 &\leq  \frac{2d_I}{n} \int_\Omega   \left|\Delta I\right|^2  \left|\nabla I\right|^2
  +\frac{n \beta^2}{2d_I}\|I\|_{L^\infty}^{2 p }\int_\Omega S^{2 q}\left|\nabla I\right|^2 \\
  &\ + \frac{d_I}{2}\int_\Omega \left|\nabla \left|\nabla I\right|^2\right|^2 +\frac{2\beta^2}{d_I}\|I\|_{L^\infty}^{2 p }\int_\Omega S^{2 q} \left|\nabla I\right|^2  \\
 &\leq  2 d_I\int_\Omega   \left|D^2 I\right|^2 \left|\nabla I\right|^2 + \frac{d_I}{2}\int_\Omega \left|\nabla \left|\nabla I\right|^2\right|^2  + \frac{(n+4)\beta^2}{2d_I}\|I\|_{L^\infty}^{2 p } \int_\Omega S^{2 q}\left|\nabla I\right|^2.
 \end{split}
 \ee
Now, we claim, for $n=1,2$ and for any $\epsilon, \eta>0$, there exist $C_2,\ C_{\epsilon, \eta}>0$ such that
\be\label{s-I-couple-12d}
\int_\Omega S^{2 q}| \nabla   I|^2
 \leq C_2 \epsilon \|\nabla |\nabla I|^2\|_{L^2}^2 + C_2 \epsilon+\eta\|\nabla S\|_{L^2}^2 + C_{\epsilon, \eta}.
\ee
We  distinguish the case of $n=1$ and $n=2$ to prove this claim.

\textbf{Case of $n=1$}. Given the $(L^1, L^2)$-boundness of $(S, \nabla I)$ in \eqref{SlnS-grad I-est}, upon twice applications of the 1-D G-N interpolation inequality, we find   that
\be\label{S-I-by-grads-gradI-1d}
\left\{\forall \eta>0,\   r\in(1,\infty],  \  s<\frac{3r}{r-1}\right\}\Longrightarrow\begin{cases}
\|S\|_{L^r}^s\leq \eta\|\nabla S\|_{L^2}^2+C_\eta, \\[0.25cm]
\left\|\left|\nabla I\right|^2\right\|_{L^r}^\frac{3r}{r-1}\leq C_3 \left\|\nabla \left|\nabla I\right|^2\right\|_{L^2}^2 + C_3.
\end{cases}
\ee
Recall that $ q<\frac{3}{2}$ forced by \eqref{kappa-lambda-con},   we then can fix $r\in (1, \infty)$ such that
$$
1<r<\frac{1}{(1-2 q)^+} \Longleftrightarrow \left\{\frac{2 q r}{r-1}:=s>1, \ \ 2q<\frac{6 q r}{2r+1}<\frac{3s}{s-1}\right\};
$$
with such chosen  $r$, we deduce from  \eqref{S-I-by-grads-gradI-1d}, H\"{o}lder and Young's inequality with epsilon, for any $\epsilon, \eta>0$,  that
\begin{align*}
\int_\Omega S^{2 q} \left| \nabla I\right|^2 & \leq \left\|S\right\|_{L^\frac{2 q r}{r-1}}^{2 q}\left\|\left|\nabla I\right|^2\right\|_{L^r}\\
&\leq \epsilon \left\|\left|\nabla I\right|^2\right\|_{L^r}^\frac{3r}{r-1} + C_\epsilon \left\|S\right\|_{L^\frac{2 q r}{r-1}}^\frac{6 q r}{2r+1}\\
&\leq C_4 \epsilon \left\|\nabla \left|\nabla I\right|^2\right\|_{L^2}^2 + C_4\epsilon+\eta \left\|\nabla S\right\|_{L^2}^2+C_{\epsilon, \eta},
\end{align*}
 yielding trivially \eqref{s-I-couple-12d} in the case of $n=1$. In particular, this implies
\begin{align}
\label{2022-1}
\int_\Omega S^{2} \left| \nabla I\right|^2 \leq C_4 \epsilon \left\|\nabla \left|\nabla I\right|^2\right\|_{L^2}^2 + C_4\epsilon+\eta \left\|\nabla S\right\|_{L^2}^2+C_{\epsilon, \eta}.
\end{align}

\textbf{Case of $n=2$}. The Young's inequality with epsilon entails
\be\label{com-est-2}
\int_\Omega S^{2 q}| \nabla   I|^2\leq \epsilon \int_\Omega  |\nabla   I|^6 + C_\epsilon \int_\Omega S^{3 q},\quad \quad \forall \epsilon>0.
\ee
In view of the  boundedness of $\|\nabla I\|_{L^2}$ by \eqref{SlnS-grad I-est}, the 2-D G-N inequality shows
\be\label{com-est-3}
\begin{split}
 \int_\Omega  \left|\nabla  I\right|^6=\left\| \left|\nabla I\right|^2\right\|_{L^3}^3 & \leq  C_5 \left(\left\| \nabla \left|\nabla I\right|^2\right\|_{L^2}^2 \left\|\left|\nabla I\right|^2\right\|_{L^1}  + \left\|\left|\nabla I\right|^2\right\|_{L^1}^3\right) \\
 &\leq C_6 \left\|\nabla \left|\nabla I\right|^2\right\|_{L^2}^2 + C_6.
 \end{split}
\ee
 Using the 2-D G-N  inequality  and noting that $ q<1$ implied by \eqref{kappa-lambda-con}, we can easily infer
\be\label{Sl3-bdd by}
\int_\Omega  S^{3 q}\leq \eta \int_\Omega |\nabla S |^2+C_\eta, \quad \quad \forall \eta>0.
\ee
Substituting \eqref{com-est-3} and \eqref{Sl3-bdd by} into \eqref{com-est-2}, we discover that \eqref{s-I-couple-12d}  holds for  $n=2$ as well.

Thus, by suitably choosing $\epsilon>0$, we derive from \eqref{gradI-l4-id}, \eqref{com-est-4}, \eqref{com-est-1} and \eqref{s-I-couple-12d} that
\begin{align}
&\frac{1}{2}\frac{d}{dt} \int_\Omega |\nabla I|^4 + \frac{d_I}{3} \int_\Omega \left|\nabla \left|\nabla I\right|^2\right|^2+ 2(\gamma+\mu)\int_\Omega \left|\nabla I\right|^{4}
\leq \eta \int_{\Omega}|\nabla S|^2 + C_\eta.
\label{2022-2}
\end{align}

Now, we shall estimate the first item on the right-hand-side of \eqref{S-l2-id}. The case of $n=1$ has been treated in \eqref{2022-1}. If $n=2$, instead of \eqref{Sl3-bdd by}, it holds
\begin{align}
\int_\Omega  S^{3}\leq C_7 \int_\Omega |\nabla S |^2 + C_7.
\label{2022-3}
\end{align}
Then, in light of \eqref{2022-3} and \eqref{com-est-3}, for any $\epsilon>0$, similar to \eqref{com-est-2}, we have
\begin{align}
\int_\Omega S^2 \left|\nabla I\right|^2 &\leq \epsilon\int_\Omega S^3 + C_\epsilon \int_\Omega \left| \nabla I\right|^6 \nonumber\\
& \leq C_7 \epsilon \int_\Omega \left|\nabla S\right|^2 + C_7 \epsilon + C_\epsilon  \left\|\nabla \left|\nabla I\right|^2\right\|_{L^2}^2 + C_7.
\label{2022-4}
\end{align}
Recalling \eqref{2022-1}, we conclude that \eqref{2022-4} still holds if $n=1$. Inserting this inequality into \eqref{S-l2-id} and choosing $\epsilon>0$ appropriately, we arrive at
\begin{align}
\frac{1}{2}\frac{d}{dt}\int_\Omega S^2+ \frac{d_S}{2}\int_\Omega|\nabla S|^2+\beta \int_\Omega S^{1+ q} I^ p \leq C_8 \int_\Omega \left| \nabla \left|\nabla I\right|^2 \right| +C_8.
\label{2022-5}
\end{align}
Multiplying \eqref{2022-2} by $\frac{6 C_8}{d_I}$ and combining \eqref{2022-5} so that the first term on the right-hand-side of \eqref{2022-5} is annihilated and then choosing $\eta>0$ with $\frac{d_S}{2} - \frac{6C_8}{d_I}\eta = \frac{d_S}{4}$, we find
\begin{align}
\frac{d}{dt}\int_\Omega \left(S^2 + \left|\nabla I\right|^4 \right) + C_9\int_\Omega \left(\left|\nabla S\right|^2 + S^{1+q}I^p + \left|\nabla I\right|^4 +
\left|\nabla \left|\nabla I\right|^2 \right| \right) \leq C_{10}.
\label{2022-6}
\end{align}
Finally, in light of the consequence of $n$-D G-N inequality, it follows readily  that
$$
\forall \, \epsilon>0, \ \exists \, C_\epsilon>0 \ \text{ s.t. } \int_\Omega S^2\leq \epsilon \int_\Omega |\nabla S|^2+C_\epsilon,
$$
we derive from \eqref{2022-6} a key differential inequality as follows:
\begin{align*}
\frac{d}{dt}\int_\Omega \left(S^2 + \left|\nabla I\right|^4 \right) + C_{11}\int_\Omega \left(S^2 + \left|\nabla S\right|^2 + S^{1+q}I^p + \left|\nabla I\right|^4 +
\left|\nabla \left|\nabla I\right|^2 \right| \right) \leq C_{12}.
\end{align*}
which, upon being solved, yields readily \eqref{S-l2-grad I-l4-est}.
\end{proof}

Armed with the global boundedness information provided by Lemmas \ref{SlnS-gradI-bdd} and \ref{L2-bdd}, we are now at the  position to present the uniform  boundedness of $\|S\|_{L^\infty}$ and $\|I\|_{W^{1,\infty}}$ and thus the global existence for \eqref{SIS-mass-C} in  dimensions one and two.

\begin{theorem}
\label{glob bdd3}
Let $\Omega\subset\mathbb{R}^n \ (n=1,2)$ be a bounded and smooth domain and, let \eqref{kappa-lambda-con} hold. Suppose the initial data fulfills \eqref{initial data-req}. Then the unique classical solution of \eqref{SIS-mass-C}  exists globally in time and is uniformly bounded in the sense   of \eqref{S-I-bdd-fin1}.
\end{theorem}

\begin{proof}
First, we utilize the $(L^2, L^\infty)$-boundedness of $(S, I)$ in \eqref{S-l2-grad I-l4-est} and  \eqref{I-linfty-un-bdd},  the $L^r$-$L^s$ type smoothing estimates in Lemma \ref{semigroup} and the fact $n q<2$ implied by \eqref{kappa-lambda-con} to deduce from the variation-of-constants formal for $I$ in  \eqref{I-rewriting} that, for all $t\in (0,T_m)$,
\be\label{gradI-est+}
\begin{split}
\|\nabla I(\cdot,t)\|_{L^\infty}&\leq \|\nabla e^{t(d_I\Delta-(\gamma+\mu))}I_0\|_{L^\infty}+\beta\int_0^t\|\nabla  e^{(t-s)(d_I\Delta-(\gamma+\mu))}S^ q I^ p \|_{L^\infty} ds\\
   &\leq C\|\nabla I_0\|_{L^\infty}+C\int_0^t\left[1+(t-s)^{-\frac{1}{2}
   -\frac{n q}{4}}\right]e^{-(\gamma+\mu)(t-s)}\|S^ q I^ p \|_{L^\frac{2}{ q}}ds\\
   &\leq C+C\int_0^t\left[1+(t-s)^{
   -\frac{2+n q}{4}}\right]e^{-(\gamma+\mu)(t-s)}\|S\|_{L^2}^ q \|I\|_{L^\infty}^ p  ds\\
   &\leq C+C\int_0^t\left[1+(t-s)^{
   -\frac{2+n q}{4}}\right]e^{-(\gamma+\mu)(t-s)} ds\\
   &\leq C+C\int_0^\infty\left(1+\tau^{
   -\frac{2+n q}{4}}\right)e^{-(\gamma+\mu)\tau}d\tau \leq C.
\end{split}
\ee
Next, we use  the variation-of-constants formula to the $S$-equation in \eqref{SIS-mass-C} to write  $S$ as
 \be\label{S-rewriting} \begin{split}
 S(t)&=e^{t(d_S\Delta-1)}S_0 +\chi\int_0^t e^{(t-s)(d_S\Delta-1)}\nabla\cdot(S\nabla I) ds\\
 &\ \  -\beta \int_0^t e^{(t-s)(d_S\Delta-1)}S^ q I^ p   ds+\int_0^t e^{(t-s)(d_S\Delta-1)}(\gamma  I+S)  ds.
 \end{split}
 \ee
 \textbf{Case of $n=1$:} Applying the  smoothing estimates of the Neumann heat semigroup on \eqref{S-rewriting}, using the obtained  $(L^2, W^{1,\infty})$-boundedness of $(S, I)$ in \eqref{S-l2-grad I-l4-est}, \eqref{I-linfty-un-bdd},  \eqref{gradI-est} and the fact that $ q<\frac{3}{2}$ by \eqref{kappa-lambda-con}, we  see, for all $t\in (0,T_m)$,
 \begin{align}
 \|S(\cdot,t)\|_{L^\infty}& \leq \left\|e^{t(d_S\Delta-1)}S_0 \right\|_{L^\infty} + \chi\int_0^t\left\|e^{(t-s)(d_S\Delta-1)}\nabla\cdot \left(S\nabla I\right)\right\|_{L^\infty}ds \nonumber \\
 &\ \quad  + \beta \int_0^t\left\|e^{(t-s)(d_S\Delta-1)}S^ q I^ p \right\|_{L^\infty} ds+ \int_0^t \left\|e^{(t-s)(d_S\Delta-1)}\left(\gamma  I+S\right)\right\|_{L^\infty} ds \nonumber \\
 &\leq C + C\int_0^t\left[1+(t-s)^{-\frac{1}{2}- \frac{1}{4}}\right]e^{-(t-s)}\left\|S\nabla I\right\|_{L^2}ds \nonumber \\
 &\ \quad +C\int_0^t\left[1+(t-s)^{-\frac{ q}{4}}\right]e^{-(t-s)} \left\|S^ q I^ p \right\|_{L^\frac{2}{ q}}ds \nonumber \\
 &\ \quad +C\int_0^t\left[1+(t-s)^{-\frac{1}{4}}\right]
 e^{-(t-s)}\left\|\gamma I+ S \right\|_{L^2}ds \nonumber \\
 &\leq C+C\int_0^t\left[1+(t-s)^{-\frac{3}{4}}\right]e^{-(t-s)}\|S\|_{L^2}\|\nabla I\|_{L^\infty}ds \nonumber \\
 &\ \quad +C\int_0^t\left[1+(t-s)^{-\frac{ q}{4}}\right]e^{-(t-s)}\|S\|_{L^2}^ q \|I\|_{L^\infty}^ p  ds \nonumber \\
 &\ \quad +C\int_0^t\left[1+(t-s)^{-\frac{1}{4}}\right]
 e^{-(t-s)}\left(\|I\|_{L^2}+\|S\|_{L^2}\right)ds \nonumber \\
 &\leq C+C\int_0^\infty\left(1+\tau^{-\frac{3}{4}}+\tau^{-\frac{ q}{4}} +\tau^{-\frac{1}{4}}\right)e^{-\tau}d\tau  \leq C.    \label{S-linfty-bdd-1d}
 \end{align}

 \textbf{Case of $n=2$:} Unlike \textbf{Case of $n=1$}, we need to further improve the $L^2$-regularity of $S$ to its $L^3$-regularity. To this purpose,  by the  $L^r$-$L^s$ type smoothing estimates of the Neumann heat semigroup, we use the established $(L^2, W^{1,\infty})$-boundedness of $(S, I)$ and the fact $ q<1$ implied by \eqref{kappa-lambda-con} to deduce first from \eqref{S-rewriting} that
 \begin{align}
 \|S(\cdot,t)\|_{L^3}& \leq \left\|e^{t(d_S\Delta-1)}S_0\right\|_{L^3} + \chi\int_0^t\left\|e^{(t-s)(d_S\Delta-1)}\nabla\cdot(S\nabla I)\right\|_{L^3}ds \nonumber \\
 &\ \quad  +\beta \int_0^t\left\|e^{(t-s)(d_S\Delta-1)}S^ q I^ p \right\|_{L^3} ds+ \int_0^t \left\|e^{(t-s)(d_S\Delta-1)}\left(\gamma  I+S\right)\right\|_{L^3} ds \nonumber \\
 &\leq C + C\int_0^t \left[1+(t-s)^{-\frac{1}{2}-\left(\frac{1}{2}-\frac{1}{3}\right)}\right]e^{-(t-s)}\left\|S\nabla I\right\|_{L^2}ds \nonumber \\
 &\ \quad +C\int_0^t\left[1+(t-s)^{-\left(\frac{ q}{2}-\frac{1}{3}\right)}\right]e^{-(t-s)}\left\|S^ q I^ p \right\|_{L^\frac{2}{ q}}ds \nonumber \\
 &\ \quad +C\int_0^t\left[1+(t-s)^{-(\frac{1}{2}-\frac{1}{3})}\right]
 e^{-(t-s)}\left\|\gamma I+S\right\|_{L^2}ds \nonumber \\
 &\leq C+C\int_0^t\left[1+(t-s)^{-\frac{2}{3}}\right]e^{-(t-s)}\|S\|_{L^2}\|\nabla I\|_{L^\infty}ds \nonumber \\
 &\ \ +C\int_0^t\left[1+(t-s)^{-\frac{3 q-2}{6}}\right]e^{-(t-s)}\|S\|_{L^2}^ q \|I\|_{L^\infty}^ p  ds \nonumber\\
 &\ \ +C\int_0^t\left[1+(t-s)^{-\frac{1}{6}}\right]
 e^{-(t-s)}\left(\|I\|_{L^2}+\|S\|_{L^2}\right)ds \nonumber \\
 &\leq C+C\int_0^\infty\left(1+\tau^{-\frac{2}{3}}+\tau^{-\frac{3 q-2}{6}}
 +\tau^{-\frac{1}{6}}\right)e^{-\tau}d\tau \nonumber\\
 &\leq C. \label{S-l3-bdd}
 \end{align}
 Then, using the just obtained  $(L^3, W^{1,\infty})$-boundedness of $(S, I)$  and the fact $ q<1$, similar  to \eqref{S-l3-bdd}, we finally conclude that
 \begin{align}
 \|S(\cdot,t)\|_{L^\infty}& \leq \left\|e^{t(d_S\Delta-1)}S_0\right\|_{L^\infty} + \chi\int_0^t\left\|e^{(t-s)(d_S\Delta-1)}\nabla\cdot(S\nabla I)\right\|_{L^\infty}ds \nonumber \\
 &\ \quad  +\beta \int_0^t\left\|e^{(t-s)(d_S\Delta-1)}S^ q I^ p \right\|_{L^\infty} ds+ \int_0^t \left\|e^{(t-s)(d_S\Delta-1)}\left(\gamma  I+S\right)\right\|_{L^\infty} ds \nonumber \\
 &\leq C+C\int_0^t\left[1+(t-s)^{-\frac{1}{2}- \frac{1}{3}}\right]e^{-(t-s)}\left\|S\nabla I\right\|_{L^3}ds \nonumber \\
 &\ \quad +C\int_0^t\left[1+(t-s)^{-\frac{ q}{3}}\right]e^{-(t-s)}\|S^ q I^ p \|_{L^\frac{3}{ q}}ds \nonumber \\
 &\ \quad +C\int_0^t\left[1+(t-s)^{-\frac{1}{3}}\right]
 e^{-(t-s)}\left\|\gamma I+S \right\|_{L^3}ds \nonumber \\
 &\leq C+C\int_0^t\left[1+(t-s)^{-\frac{5}{6}}\right]e^{-(t-s)}\|S\|_{L^3}\|\nabla I\|_{L^\infty}ds \nonumber\\
 &\ \quad  +C\int_0^t\left[1+(t-s)^{-\frac{ q}{3}}\right]e^{-(t-s)}\|S\|_{L^3}^ q \|I\|_{L^\infty}^ p  ds \nonumber \\
 &\ \quad +C\int_0^t\left[1+(t-s)^{-\frac{1}{3}}\right]
 e^{-(t-s)}\left(\|I\|_{L^3}+\|S\|_{L^3}\right)ds \nonumber \\
 &\leq C+C\int_0^\infty\left(1+\tau^{-\frac{5}{6}}+\tau^{-\frac{ q}{3}}
 +\tau^{-\frac{1}{3}}\right)e^{-\tau}d\tau \nonumber\\
 &\leq C. \label{S-linfty-bdd}
 \end{align}
Combining  \eqref{I-linfty-un-bdd},  \eqref{S-linfty-bdd-1d},  \eqref{S-linfty-bdd} and \eqref{S-bdd-imply gradIbdd0}, we achieve our desired uniform-in-time  estimate \eqref{S-I-bdd-fin1} on $[0, T_m)$. Consequently, the extensibility criterion  \eqref{extend} in the local existence of Lemma \ref{local-in-time} yields first  that $T_m=\infty$ and then the boundedness \eqref{S-I-bdd-fin1}  for all $t>0$.
\end{proof}

\section{Long time behavior of global bounded solutions}
 For any bounded global classical solution $(S, I)$ of \eqref{SIS-mass-C} obeying \eqref{S-I-bdd-fin1}, applying the H\"{o}lder estimates for parabolic  equations (cf. \cite[Theorem 1.3]{PV93-JDE})  and then using  the  standard parabolic Schauder theory (cf. \cite{Fried, La}) repeatedly (see similar argument in \cite{LPX18}), we see  there exist  $\theta\in(0, 1)$ and $C_0>0$ such that
 \be\label{reg-SI}
 \|S\|_{C^{2+\theta, 1+\frac{\theta}{2}}(\overline{\Omega}\times[t,t+1])}+\|I\|_{C^{2+\theta, 1+\frac{\theta}{2}}(\overline{\Omega}\times[t,t+1])}\leq C_0, \quad \forall \ t\geq 1.
 \ee

\subsection{ With cross-diffusion $\chi\neq 0$ and with motality $\mu>0$}
Our main result in this subsection reads as follows.

\begin{theorem}
\label{thm-lt1}In the case of $\mu>0$, the cross-diffusive SIS model \eqref{SIS-mass-C} does not admit threshold dynamics. More precisely, for any bounded global classical solution $(S, I)$ of \eqref{SIS-mass-C} obeying \eqref{S-I-bdd-fin1}, it follows, as $t\rightarrow \infty$, that
\be\label{S-con}
\left(S(\cdot,t),I(\cdot,t)\right)\rightarrow(S_*,0)
\ee
uniformly on $\overline{\Omega}$, where
\be\label{S*-def}
\begin{cases}S_*=0, &\text{ if }  0<p<1, \\[0.2cm]
0<S_*=|\Omega|^{-1}\left(\int_\Omega \left(S_0+I_0\right)-\mu\int_0^\infty\int_\Omega I\right), &\text{ if }  p\geq 1.
\end{cases}
\ee
Furthermore, $S_*\leq \left(\frac{\gamma + \mu }{\beta}\right)^{\frac{1}{q}}$ if $p=1$, and $(S, I)\rightarrow (S_*,0)$ exponentially if $p>1$.
\end{theorem}

\begin{proof}
An  integration of \eqref{S+I-id} shows, for all $t>0$, that
\be\label{S+I-int-id}
 \int_\Omega \left(S+I\right)+\mu \int_0^t\int_\Omega I= \int_\Omega \left(S_0+I_0\right),
\ee
which along with the fact $\mu>0$ implies that
\be\label{glo-dyn7}
\int_0^\infty\int_\Omega I\leq \frac{1}{\mu}\int_\Omega \left(S_0+I_0\right)<\infty.
\ee
By the regularity \eqref{reg-SI}, the integrand $\int_\Omega I$ is uniformly bounded and uniformly continuous, and so $\|I(\cdot,t)\|_{L^1}\rightarrow 0$ as $t\rightarrow \infty$.  This, combined with \eqref{reg-SI} and the standard embedding theorem (cf. \eqref{SLinfty-gn-l2} below or \cite{Xiang16-ppt, Xiangpre2} for instance), yields that
\begin{equation}
I(\cdot,t)\to 0{\rm \ uniformly \ on \ }\overline\Omega, \quad {\rm as \ }t\to\infty.
\label{glo-dyn1}
\end{equation}
As a result, since $\int_\Omega (S+I)$ is decreasing and is nonnegative, we can assume
\begin{equation}
\bar{S}:=\frac{1}{|\Omega|}\int_{\Omega}S(\cdot,t) \to S_*,\quad {\mbox as \ }t\to\infty,
\label{glo-dyn2}
\end{equation}
for some nonnegative constant $S_*$.

Now, we claim that
\begin{align}
S(\cdot,t)\to S_*{\rm \ uniformly \ on \ }\overline\Omega, \quad {\rm as \ }t\to\infty.
\label{glo-dyn3}
\end{align}
To this aim, we notice from the $S$-equation in \eqref{SIS-mass-C} and \eqref{reg-SI} that
\begin{align}
\frac{1}{2}\frac{d}{dt}\int_\Omega S^2 & = -d_S \int_\Omega |\nabla S|^2 - \chi \int_\Omega S\nabla I\cdot \nabla S - \beta \int_\Omega S^{q+1} I^p + \gamma \int_\Omega SI \nonumber\\
&\leq - \frac{d_S}{2} \int_\Omega |\nabla S|^2  + \frac{\chi^2}{2d_S}\int_\Omega S^2 |\nabla I|^2  + \gamma \int_\Omega SI \nonumber\\
&\leq - \frac{d_S}{2} \int_\Omega |\nabla S|^2 + \frac{C_0^2 \chi^2}{2d_S}\int_\Omega |\nabla I|^2 + \gamma C_0 \int_\Omega I,\quad \forall \,t\geq1.
\label{glo-dyn4}
\end{align}
Similarly, one can readily  obtain from the $I$-equation that
\begin{align*}
\frac{1}{2}\frac{d}{dt}\int_\Omega I^2 & = -d_I \int_\Omega |\nabla I|^2 + \beta \int_\Omega S^q I^{p+1}- (\gamma+\mu) \int_{\Omega}I^2 \nonumber\\
&\leq  -d_I \int_\Omega |\nabla I|^2 + \beta C_0^{p+q} \int_\Omega  I,\quad \forall \, t\geq 1.
\end{align*}
Upon an integration and a use of \eqref{glo-dyn7}, we find that
\begin{align}
   \int_1^\infty \int_\Omega |\nabla I|^2 &\leq \frac{\beta C_0^{p+q}}{d_I} \int_1^\infty \int_\Omega I + \frac{1}{2d_I}\int_\Omega I^2(\cdot,1)   \nonumber\\
& \leq \frac{\beta C_0^{p+q}}{\mu d_I}\int_\Omega \left(S_0+I_0\right)+ \frac{1}{2d_I}C_0^2 |\Omega|<\infty.
\label{glo-dyn6}
\end{align}
 Integrating \eqref{glo-dyn4} and using \eqref{glo-dyn7} and \eqref{glo-dyn6}, we deduce, for some $C_1>0$ that
 $$
 \int_1^\infty \int_\Omega |\nabla S|^2 \leq C_1<\infty.
 $$
 This along with the uniform continuity of $ \int_\Omega |\nabla S|^2$ due to \eqref{reg-SI} shows
\begin{align}
\int_\Omega |\nabla S(\cdot,t)|^2 \to 0,\quad \mbox{as } t\to \infty.
\nonumber
\end{align}
Hence, we conclude from \eqref{glo-dyn2} and  the  Poincar\'{e}  inequality that
\be\label{S-bar-poincare}
\begin{split}
\int_\Omega \left| S(\cdot,t) - S_*\right|^2&=\left| \bar{S} - S_*\right|^2|\Omega|+\int_\Omega \left| S(\cdot,t) - \bar{S}\right|^2\\
 &\ \ \leq \left| \bar{S} - S_*\right|^2|\Omega|+C\int_\Omega \left| \nabla S(\cdot,t)\right|^2\to 0,\ \  \mbox{as } t\to \infty,
\end{split}
\ee
which along with \eqref{reg-SI} and the G-N inequality in \eqref{G-N-I} shows
\be\label{SLinfty-gn-l2} \begin{split}
\|S(\cdot,t)-S_*\|_{L^\infty}&\leq C\left(\|\nabla S(\cdot,t)\|_{L^\infty}^\frac{n}{n+2}+\| S(\cdot,t)-S_*\|_{L^2}^\frac{n}{n+2}\right)\| S(\cdot,t)-S_*\|_{L^2}^\frac{2}{n+2}\cr
&\ \leq C \| S(\cdot,t)-S_*\|_{L^2}^\frac{2}{n+2}\rightarrow 0\ \ \text{as} \ \ t\rightarrow \infty.
\end{split}
\ee
This precisely gives  rise to \eqref{glo-dyn3}.

(i) We now show that $S_*=0$ if $p<1$. Suppose on the contrary that $S_*>0$. Then it follows from \eqref{glo-dyn1} and \eqref{glo-dyn3} there exists some $T_0>0$ such that
$$\beta S^q(x,t) - (\gamma+\mu)I^{1-p}(x,t) > 0,\quad   (x,t)\in\Omega \times (T_0,\infty).$$
As a result, we have
\begin{align}
I_t - d_I \Delta I = I^p \left[ \beta S^q - (\gamma+\mu)I^{1-p}\right]>0,\quad (x,t)\in\Omega\times (T_0,\infty).
\nonumber
\end{align}
A simple comparison argument then yields $I(x,t) \geq \min_{x\in \overline\Omega}I(x,T_0)>0$ for all $(x,t)\in \overline\Omega \times [T_0,\infty)$, contradicting \eqref{glo-dyn1}. Thus, it must hold $S_* = 0$ if $p<1.$

(ii) In the sequel, we aim to prove that $S_*>0$, provided $p\geq 1$. Suppose on the contrary that $S_*=0$.
Then there exists some $T_1>1$ fulfilling
\begin{align}
\beta S^q(x,t) I^{p-1}(x,t) - (\gamma+\mu) \leq -\frac{1}{2}(\gamma+\mu),\quad   (x,t)\in \overline\Omega \times [T_1,\infty).
\label{glo-dyn8}
\end{align}
and
\begin{align}
\gamma - \beta S^q(x,t) I^{p-1}(x,t) >0,\quad   (x,t)\in \overline\Omega \times [T_1,\infty).
\label{glo-dyn8.5}
\end{align}
Thanks to \eqref{glo-dyn8}, one can easily see from the $I$-equation in \eqref{SIS-mass-C} that $I$ satisfies
\begin{equation}
\left\{ \begin{array}{ll}
\displaystyle I_t \leq d_I \Delta I -\frac{1}{2}(\gamma+\mu)I,&x\in\Omega,\, t>T_1,\\
\noalign {\vskip 4pt}
\displaystyle \frac{\partial I}{\partial \nu} = 0, &x\in\partial\Omega, \, t>T_1.
\end{array}\right.
\nonumber
\end{equation}
Consider the following corresponding ODE problem
\begin{equation}
\left\{ \begin{array}{ll}
\displaystyle \frac{d w}{dt} = -\frac{1}{2}(\gamma+\mu)w,& t>T_1,\\
\noalign {\vskip 4pt}
\displaystyle w(T_1) = \max_{x\in\overline\Omega}I(x,T_1).
\end{array}\right.
\nonumber
\end{equation}
An application of the comparison principle yields
\begin{align}
0<I(x,t) \leq w(t) = \max_{x\in\overline\Omega}I(x,T_1)e^{-\frac{1}{2}(\gamma+\mu)(t-T_1)}, \quad (x,t)\in\overline\Omega\times [T_1,\infty).
\nonumber
\end{align}
This yields with a convenient constant $C_2>0$ that
\begin{align}
0<I(x,t) \leq C_2 e^{-\frac{1}{2}(\gamma+\mu)t}, \quad (x,t)\in\overline\Omega\times [T_1,\infty).
\label{glo-dyn11}
\end{align}
Lemma \ref{local-in-time} and the regularity  \eqref{reg-SI} imply  that   $f(x,t)=\beta\Delta\left(S^qI^p\right)\in C^0(\overline{\Omega}\times[T_1, \infty))$. Then, writing  $z=\Delta I$, we infer from the $I$-equation in \eqref{SIS-mass-C} that
\begin{align*}
\begin{cases}
z_t=d_I\Delta z-(\gamma+\mu)z+f(x,t),    &(x,t)\in \Omega\times [T_1,\infty),\\ \noalign {\vskip 4pt}
\frac{\partial z}{\partial \nu}=0,  &(x,t)\in \partial\Omega\times [T_1,\infty),
\end{cases}
\end{align*}
which along with the standard Schauder estimate shows that $\|z(\cdot,t)\|_{C^2}$ and thus $\|D^3I(\cdot,t)\|_{L^\infty}$ is uniformly bounded due to the elliptic estimate.

In view of the exponential convergence of $I$ in  \eqref{glo-dyn11}, we  make use of the G-N inequality in Lemma \ref{GNI} to obtain, for some $C_4,\, C_5>0$, that
\be\begin{split}
\label{I-w2-decay}
\|D^2I(\cdot,t)\|_{L^\infty(\Omega)} &\leq C_3 \left(\|D^3I(\cdot,t)\|_{L^\infty(\Omega)}^{2/3} \|I(\cdot,t)\|_{L^\infty(\Omega)}^{1/3} + \|I(\cdot,t)\|_{L^\infty(\Omega)}  \right)   \\
& \leq C_4 e^{-C_5t},\quad t\geq T_1.
\end{split}
\ee
In particular, this indicates that
\begin{align}
\Delta I(x,t)\geq -C_4 e^{-C_5t}, \quad (x,t)\in \overline\Omega \times [T_1,\infty).
\label{glo-dyn12}
\end{align}
Now, observe from the $S$-equation in \eqref{SIS-mass-C} that
\begin{align}
S_t & = d_S \Delta S + \chi \nabla I \cdot \nabla S + \chi S \Delta I + I \left(\gamma - \beta S^q I^{p-1}\right) \nonumber\\
& \geq d_S \Delta S + \chi \nabla I \cdot \nabla S + \chi S \Delta I \nonumber\\
& \geq d_S \Delta S + \chi \nabla I \cdot \nabla S - C_4 \chi e^{-C_5 t}S,
\nonumber
\end{align}
for $x\in\Omega$ and $t>T_1$, where we have used \eqref{glo-dyn8.5} and \eqref{glo-dyn12}.
Consider the ODE problem
\begin{equation}
\left\{ \begin{array}{ll}
\displaystyle \frac{d v}{dt} = -C_4 \chi e^{-C_5 t}v,& t>T_1,\\
\noalign {\vskip 4pt}
\displaystyle v(T_1) = \min_{x\in\overline\Omega}S(x,T_1)>0.
\end{array}\right.
\nonumber
\end{equation}
We again employ the comparison principle to derive
\begin{align}
S(x,t) \geq v(t), \quad (x,t)\in \overline\Omega\times [T_1,\infty).
\nonumber
\end{align}
It can be easily verified, for $t>T_1$, that
\begin{align}
v(t) = v(T_1)\exp\left\{\frac{C_4 \chi}{C_5}e^{-C_5 T_1} \left[e^{-C_5 (t-T_1)} -1\right] \right\} \geq v(T_1)\exp\left\{-\frac{C_4 \chi}{C_5}e^{-C_5 T_1}\right\}>0.
\nonumber
\end{align}
As a result, it holds
\begin{align}
\liminf_{t\to \infty}\min_{x\in\overline\Omega}S(x,t) \geq v(T_1)\exp\left\{-\frac{C_4 \chi}{C_5}e^{-C_5 T_1}\right\}>0,
\nonumber
\end{align}
contradicting our assumption that $S_*=0$. Henceforth, we must have $S_*>0$.

(iii) We shall illustrate that $S_*^q\leq \frac{\gamma + \mu}{\beta}$ if $p=1$. We proceed on the contrary that $S_*^q>\frac{\gamma + \mu}{\beta}$ if $p=1$. Then there must exist some $T_2>1$ such that $(\beta S^q - \gamma - \mu)I>0$ on $\overline\Omega \times [T_2,\infty)$. Thus, we have
\begin{align}
I_t - d_I \Delta I>0,\quad x\in\Omega,\, t>T_2.
\nonumber
\end{align}
Again, it follows easily from  the comparison principle that $I(x,t)\geq \min_{x\in\overline\Omega}I(x,T_2)>0$ for all $(x,t) \in \overline \Omega \times [T_2,\infty)$. However, this is a contradiction to \eqref{glo-dyn1}.

Sending  $t\rightarrow \infty$ in \eqref{S+I-int-id} and applying \eqref{glo-dyn1} and \eqref{glo-dyn3}, we obtain the implicit formula for $S_*$ in \eqref{S*-def}.

(iv) We shall show that the convergence rate of $\left(S,I\right)\to (S_*,0)$ is exponential in the case of $p>1$. Since $I$ converges uniformly to zero, the above arguments indeed have shown (cf. \eqref{glo-dyn8}, \eqref{glo-dyn11} and \eqref{I-w2-decay}), for some $C_5,\, C_6>0$, that
\be\label{I-grad-decay}
\|I(\cdot, t)\|_{W^{1, \infty}}\leq C_5e^{-C_6t}, \ \ t\geq T_1.
\ee
This along with \eqref{glo-dyn2},  \eqref{S+I-int-id} and \eqref{S*-def} allows us to deduce that
\be\label{S*-Sbar}
\left|\overline{S}-S_*\right|=\frac{1}{|\Omega|}\left|\mu\int_t^\infty\int_\Omega I-\int_\Omega I\right|\leq C_7e^{-C_6t}, \ \ \ \ t\geq T_1.
\ee
Now, testing the $S$-equation in \eqref{SIS-mass-C}   by $(S-S_*)$ and using the boundedness of $(S,I)$, we find that there exists some $C_8>0$ fulfilling
\begin{align*}
&\frac{1}{2}\frac{d}{dt}\int_\Omega \left(S-S_*\right)^2\\
& = -d_S \int_\Omega |\nabla S|^2 - \chi \int_\Omega S\nabla I\cdot \nabla S + \int_\Omega \left(\gamma-\beta S^q I^{p-1} \right)\left(S-S_*\right)I\\
&\leq - \frac{d_S}{2} \int_\Omega |\nabla S|^2  + C_8 \|\nabla I(\cdot, t)\|_{L^\infty}^2 + C_8\|I(\cdot, t)\|_{L^\infty},\quad \forall \,t\geq T_1,
\end{align*}
which along with \eqref{S-bar-poincare}, \eqref{I-grad-decay} and \eqref{S*-Sbar} enables us to derive a differential inequality for $\|S-S_*\|_{L^2}^2$ as follows: for some $a,\, \delta,\,  C_9>0$ with $\delta<a$,
$$
\frac{d}{dt}\int_\Omega \left(S-S_*\right)^2+a\int_\Omega \left(S-S_*\right)^2\leq C_9e^{-\delta t}, \ \ t\geq T_1.
$$
This directly entails
$$
\int_\Omega \left(S(\cdot,t)-S_*\right)^2\leq e^{-a(t-T_1)}\int_\Omega \left(S(\cdot,T_1)-S_*\right)^2+\frac{C_9}{a-\delta}e^{-\delta t}, \ \ t\geq T_1,
$$
which, in conjunction with \eqref{SLinfty-gn-l2} and \eqref{I-grad-decay}, yields our desired exponential convergence of $(S, I)$ to $(S_*, 0)$.
The proof is thus completed.
\end{proof}

\subsection{Without cross-diffusion $\chi=0$  and without mortality $\mu=0$}

In this subsection, we shall always assume that $\chi =0$ and $\mu=0$. That is, we focus on the following reaction-diffusion SIS epidemic model
\be\label{SIS-mass}
\begin{cases}
S_t=d_S \Delta S -\beta S^ q I^ p +\gamma I,& x\in\Omega,\, t>0,\\[3pt]
I_t=d_I\Delta I+\beta  S^ q I^ p - \gamma  I,& x\in\Omega, \, t>0,\\[3pt]
\frac{\partial S}{\partial \nu}=\frac{\partial I}{\partial \nu}=0, & x\in \partial \Omega, \, t>0,\\[3pt]
S(x,0)=S_0(x),  \ \   I(x,0)= I_0(x), & x\in \Omega.
\end{cases}
\ee
From now on, we no longer assume that transmission rate $\beta$ and recovery rate $\gamma$ and positive constants.
Instead, we take spatial heterogeneity into consideration. That is, we assume $\beta(x)$ and $\gamma(x)$ are positive $C^\alpha(\overline\Omega)$-smooth functions for some $\alpha>0$.
According to \cite[Theorem 2.3]{PW19}, the model \eqref{SIS-mass} admits a unique global and bounded solution $(S,I)$ for any $p,\,q>0$ with $S(x,t),\,I(x,t)>0$ for all $x\in\overline\Omega$ and $t>0$. As a result, \eqref{reg-SI} still holds.
 In addition, it is easily seen that the total population number is conserved:
\begin{equation}
\int_\Omega (S+I) = \int_{\Omega}(S_0 +I_0) =:N>0,\quad \forall \, t>0.
\label{N}
\end{equation}

Concerning the long-time behavior of solutions, in the case of $p\in (0,1)$, the authors in \cite{PW19} only showed  the uniform persistence of solutions; if $p=1$, a basic reproduction number $\mathcal R_0$ can be defined and the uniform persistence property still holds if $\mathcal R_0>1$; if $p>1$, their analysis about the long-time behavior of the ODE version of \eqref{SIS-mass} suggests that its dynamics depend on the initial data. To the best of our knowledge, we are not aware of any further results on the global attractivity of disease-free equilibrium (DFE) or endemic equilibrium (EE) for the reaction-diffusion system \eqref{SIS-mass} presented in literature, although there are some studies \cite{Koro, Lei1} on the S(E)IR type ODE and PDE models with varying total population. Here, for $0<p\leq 1$, we shall investigate the global asymptotic stability of DFE or EE in two special cases: (i) $\gamma(x)=r \beta(x)$ for some constant $r>0$ and all $x\in \Omega$, and (ii) $d_S = d_I$.

We first consider the case of $0<p<1$.

Let $q,\,r>0$. For $S\in \left[0,\frac{N}{|\Omega|}\right]$, we define a function $$f(S) = r\left(\frac{N}{|\Omega|}-S\right)^{1-p} - S^q.$$
Then it can be readily checked that $f(0)>0$ and $f\left(\frac{N}{|\Omega|}\right)<0$. Moreover, $f'(S)<0$ for all $S\in \left(0,\frac{N}{|\Omega|}\right)$. Consequently, there exists a unique $S^*\in \left(0,\frac{N}{|\Omega|}\right)$ satisfying
\begin{equation}\label{S-p1}
 r\left(I^*\right)^{1-p}:=r\left(\frac{N}{|\Omega|}-S^*\right)^{1-p} = (S^*)^q.
\end{equation}
 Obviously, $(S^*,I^*)$ is an EE of \eqref{SIS-mass} if $\frac{\gamma(x)}{\beta(x)}\equiv r$ for all $x\in\Omega$.
\begin{theorem}
\label{thm-glo1}
Suppose $0<p<1$. Assume there exists some $r>0$ such that $\gamma(x)=r \beta(x)$ for all $x\in \Omega$. Then, as $t\to \infty$, the unique global and bounded solution $(S,I)$ of the reaction-diffusion model  \eqref{SIS-mass} satisfies
\be
\left(S(x,t),I(x,t)\right) \to (S^*,I^*),
\nonumber
\ee
uniformly for $x\in\overline\Omega$, where $(S^*, I^*)$ is defined by \eqref{S-p1}.
\end{theorem}

\begin{proof}
Suppose $q\neq 1$.
We define the following Lyapunov functional
$$V_1(S,I) = \int_\Omega\left\{ \left[\left(S-S^*\right) -\frac{(S^*)^q}{1-q} \left( S^{1-q}-(S^*)^{1-q} \right) \right] + \left[(I- I^*) - \frac{(I^*)^{1-p}}{p} \left(I^p - (I^*)^p\right)  \right]  \right\}.$$
Elementary analysis shows that $V_1(S,I) \geq 0$ for all $S,\,I>0$. Moreover, in light of $(S^*)^q = r (I^*)^{1-p}$ due to \eqref{S-p1}, by direct calculations, we compute from \eqref{SIS-mass} that
\begin{align}
\frac{d}{dt}V_1(S,I)&=\int_\Omega \left[\left(1-\frac{(S^*)^q}{S^q}  \right)S_t + \left(1-\frac{(I^*)^{1-p}}{I^{1-p}}\right)I_t \right] \nonumber\\
& = -q d_S (S^*)^q \int_\Omega \frac{|\nabla S|^2}{S^{q+1}} - d_I (1-p)(I^*)^{1-p} \int_\Omega \frac{|\nabla I|^2}{I^{2-p}} - (I^*)^{1-p} \int_\Omega \beta \frac{(S^q I^p - rI)^2}{S^q I} \nonumber\\
& \leq 0. \nonumber
\end{align}
Thus, upon an integration in $t\in (1,\infty)$ and a use of the boundedness of $(S,I)$, we find
\be
\int_1^\infty\int_\Omega \left[qr d_S \frac{|\nabla S|^2}{S^{q+1}} + d_I (1-p) \frac{|\nabla I|^2}{I^{2-p}} +  \beta \frac{(S^q I^p - rI)^2}{S^q I} \right]<\infty.
\label{thm-glo1-1}
\ee
The regularity in \eqref{reg-SI} along with Ascoli-Arzel\`{a} theorem implies there exists a sequence $t_k\rightarrow \infty$ and nonnegative $C^2$-function $(S_\infty, I_\infty)$ such that $(S(\cdot, t_k), I(\cdot, t_k))\rightarrow (S_\infty,I_\infty)$ in $C^2(\overline{\Omega})$, as $k\rightarrow \infty$. One the other hand, recall that $(S,I)$ is uniformly persistent by \cite[Theorem 2.5]{PW19}, then the function
$$t\mapsto \int_\Omega \left[qr d_S \frac{|\nabla S|^2}{S^{q+1}} + d_I (1-p) \frac{|\nabla I|^2}{I^{2-p}} +  \beta \frac{(S^q I^p - rI)^2}{S^q I} \right]$$
 is  uniformly bounded and uniformly continuous. Thus, we infer from \eqref{thm-glo1-1} that
\be
\int_\Omega \left[qr d_S \frac{|\nabla S_\infty|^2}{S_\infty^{q+1}} + d_I (1-p) \frac{|\nabla I_\infty|^2}{I_\infty^{2-p}} +  \beta \frac{(S_\infty^q I_\infty^p - rI_\infty)^2}{S_\infty^q I_\infty} \right] =0,
%\label{thm-glo1-2}
\nonumber
\ee
which entails both $S_\infty$ and $I_\infty$ are constants and $S^q_\infty I^p_\infty - r I_\infty =0$. In addition, $S_\infty$ and $I_\infty$ must be positive constants again by the uniform persistence of $(S,I)$. As a result, it must hold $(S_\infty,I_\infty) = (S^*,I^*)$ by uniqueness. Since the limiting function $(S_\infty,I_\infty)$ is unique, it follows that $(S(\cdot,t),I(\cdot,t))\to (S^*,I^*)$ in $C^2(\overline\Omega)$, as $t\to\infty$.

When  $q=1$, we  define the following Lyapunov functional
$$V_2(S,I) = \int_\Omega\left\{ \left(S-S^* - S^* \ln\frac{S}{S^*} \right)   + \left[(I- I^*) - \frac{(I^*)^{1-p}}{p} \left(I^p - (I^*)^p\right)  \right]  \right\}.$$
Then $V_2(S,I) \geq 0$ for all $S,\,I>0$. Furthermore, we have
\begin{align*}
\frac{d}{dt}V_2(S,I)&=\int_\Omega \left[\left(1-\frac{S^*}{S}  \right)S_t + \left(1-\frac{(I^*)^{1-p}}{I^{1-p}}\right)I_t \right] \nonumber\\
& = -d_S S^* \int_\Omega \frac{|\nabla S|^2}{S^{2}} - d_I (1-p)(I^*)^{1-p} \int_\Omega \frac{|\nabla I|^2}{I^{2-p}} - (I^*)^{1-p} \int_\Omega \beta \frac{(S I^p - rI)^2}{S^q I} \nonumber\\
& \leq 0.
\end{align*}
Then the  rest of the argument is the same as that  of $q\neq 1$ as done before.
\end{proof}

Now, suppose that $d_S = d_I$.  For our later purpose,  given  constant  $\tau_0>0$, we consider the following initial-boundary value problem:
\begin{equation}
\left\{ \begin{array}{lll}
u_t = d_I \Delta u + \beta(x) \left(\tau_0 - u\right)^q u^p - \gamma(x)u,&x\in\Omega,\, t>0,\\ \noalign{\vskip 3pt}
\frac{\partial u}{\partial \nu} =0, &x\in\partial \Omega,\,t>0,\\ \noalign{\vskip 3pt}
u(x,0)=u_0(x),&x\in\Omega.
\end{array}\right.
\label{thm-glo2-1}
\end{equation}

\begin{lemma}
\label{glo-lem1}
Let  $0<p<1$. Inside the interval $\left[0,\tau_0\right]$, the problem \eqref{thm-glo2-1} admits a unique positive steady state $U(x)$, which is globally asymptotically stable for \eqref{thm-glo2-1} with initial data fulfilling $\inf_{\Omega}u_0>0$ and $\|u_0\|_{L^\infty} < \tau_0$.
\end{lemma}

\begin{proof}
Since $0<p<1$,  define
$$
\delta_0=\sup\left\{\tilde\delta\in (0, \tau_0): \ \left(\tau_0-\delta\right)^q\inf_{\Omega}\frac{\beta}{\gamma}-\delta^{1-p}>0, \ \forall \delta\in(0,\tilde\delta)\right\}.
$$
Then one can easily check for any  $\delta\in (0, \delta_0]$,  that the two positive constants $\delta$ and $\tau_0$, respectively, are lower and upper solutions for the steady state problem of \eqref{thm-glo2-1}. Thus, the existence of a positive equilibrium $U(x) \in \left[0,\tau_0 \right]$ is guaranteed  by the standard  upper-lower solutions method.

Next, we first show there exists a unique steady state in $\left[\delta,\tau_0 \right]$. To this end, we let $u_1$ and $u_2$ be the minimal and maximal solutions in $\left[\delta,\tau_0 \right]$.  It then follows that $\delta \leq u_1\leq u_2\leq \tau_0$ on $\overline{\Omega}$. Then we multiply by $u_2$ the equation satisfied by $u_1$ to derive
\begin{equation}
-d_I \int_\Omega \nabla u_1 \cdot \nabla u_2 + \int_\Omega \left[ \beta \left(\tau_0 -u_1\right)^q u_1^p u_2 - \gamma u_1 u_2 \right] =0.
\label{thm-glo2-2}
\end{equation}
By interchanging the role of $u_1$ and $u_2$, we have
\begin{equation}
-d_I \int_\Omega \nabla u_1 \cdot \nabla u_2 + \int_\Omega \left[ \beta \left(\tau_0 -u_2\right)^q u_2^p u_1 - \gamma u_1 u_2 \right] =0.
\label{thm-glo2-3}
\end{equation}
Thus, we see from \eqref{thm-glo2-2} and \eqref{thm-glo2-3} that
\begin{equation}
\int_\Omega \beta u_1^p u_2^p \left[\left(\tau_0 - u_1\right)^q u_2^{1-p} - \left(\tau_0 - u_2\right)^q u_1^{1-p}  \right]=0.
%\label{thm-glo2-4}
\nonumber
\end{equation}
This along with the facts  $\left(\tau_0 - u_1\right)^q \geq \left(\tau_0 - u_2\right)^q >0 $ and $u_2^{1-p} \geq u_1^{1-p}$ shows that $u_1 \equiv u_2$ on $\overline{\Omega}$, yielding the desired uniqueness on $[\delta, \tau_0]$.  Now, for any two positive equilibria  $u_1$ and $u_2$ on $[0, \tau_0]$,  we  put $\delta=\min\{\min_{\overline{\Omega}}u_1, \ \min_{\overline{\Omega}}u_2, \ \delta_0\}$; then $\delta>0$ and $u_1, u_2$ are two equilibria on $[\delta, \tau_0]$, and so the proved uniqueness shows $u_1\equiv u_2$ on $\overline{\Omega}$.

Given initial data fulfilling $\inf_{\Omega}u_0>0$ and $\|u_0\|_{L^\infty} < \tau_0$, for any $0<\delta\leq \min\{\inf_{\Omega}u_0, \ \delta_0\}$, we see that  $ \delta \leq u_0(x) \leq \tau_0$ for all $x\in\overline\Omega$. Denoting  by $u(x,t;\delta)$ and $u(x,t;\tau_0)$, respectively, the solutions of \eqref{thm-glo2-1} with initial data $u(x,0)\equiv \delta$ and $u(x,0) \equiv \tau_0$, we then conclude from the maximum principle  that
\begin{equation}
\delta \leq u(x,t;\delta) \leq u(x,t) \leq u\left(x,t;\tau_0 \right) \leq \tau_0,
\label{thm-glo2-4}
\end{equation}
for all $(x,t) \in \overline \Omega \times (0,\infty)$. In addition, $u(x,t;\delta)$ is non-decreasing and $u(x,t;\tau_0 )$ is non-increasing with respect to $t\in (0,\infty)$. As a result, the point-wise limits
\begin{equation}
\delta\leq \lim_{t\to\infty}u(x,t;\delta) =: \underline U(x) \quad \mbox{and}\quad \lim_{t\to\infty}u\left(x,t;\tau_0\right) =: \overline U(x)\leq \tau_0,
\label{thm-glo2-4.5}
\end{equation}
exist. Furthermore, notice that both $\underline U$ and $\overline U$ are positive equilibria of \eqref{thm-glo2-1} in $(0,\tau_0)$, and thus, $\underline U\equiv \overline U$.  Hence, in light of the standard parabolic regularity as in \eqref{reg-SI} or simply  Dini's theorem,    we infer from \eqref{thm-glo2-4} and \eqref{thm-glo2-4.5} that  $u(x,t;u_0)\rightarrow U(x) \equiv \underline U(x)$   as $t\rightarrow \infty$ uniformly for $x\in \overline{\Omega}$.
\end{proof}

Now, we are ready to show the global asymptotic stability of EE when $d_S = d_I$.

\begin{theorem}
\label{thm-glo2}
 Let $0<p<1$ and $d_S =d_I$. Then, as $t\to\infty$, the unique global and bounded solution $(S,I)$ of the reaction-diffusion model \eqref{SIS-mass} fulfills
$$(S(x,t),I(x,t)) \to (\tilde S(x),\tilde I(x)),$$
uniformly for $x\in\overline\Omega$, where $(\tilde S(x),\tilde I(x))$ is the unique endemic equilibrium of \eqref{SIS-mass}.
\end{theorem}

\begin{proof}
Let $w(x,t) = S(x,t) + I(x,t)$. Obviously, $w$ satisfies
\begin{equation}
\left\{ \begin{array}{llll}
w_t = d_S \Delta w,&x\in\Omega,\,t>0,\\ \noalign {\vskip 4pt}
\frac{\partial w}{\partial \nu} =0,&x\in\partial\Omega,\,t>0,\\ \noalign {\vskip 4pt}
w(x,0) =S_0(x) + I_0(x),&x\in\Omega,\\ \noalign {\vskip 4pt}
\int_\Omega w =N,&t>0.
\end{array}\right.
\nonumber
\end{equation}
Therefore, as $t\to\infty$, $w(x,t) \to \frac{N}{|\Omega|}$ uniformly for $x\in\overline\Omega$. So,  for any small $\varepsilon>0$, there exists some $T=T(\varepsilon)>0$ such that
\begin{equation}
\frac{N}{|\Omega|}- I(x,t) - \varepsilon < S(x,t) < \frac{N}{|\Omega|}- I(x,t) + \varepsilon,
\label{thm-glo2-5}
\end{equation}
for all $x\in\overline\Omega$ and $t\geq T$. Consider the auxiliary problems
\begin{equation}
\left\{ \begin{array}{lll}
\bar I_t - d_I \Delta \bar I  = \beta \left(\frac{N}{|\Omega|}+\varepsilon - \bar I \right)^q \bar I^p - \gamma \bar I,&x\in\Omega,\,t>T,\\ \noalign{\vskip 3pt}
\frac{\partial \bar I}{\partial \nu} =0,&x\in\partial\Omega,\,t>T, \\ \noalign{\vskip 3pt}
\bar I(x,T) = \max_{x\in\overline\Omega}I(x,T),&x\in\Omega,
\end{array}\right.
\label{thm-glo2-6}
\end{equation}
and
\begin{equation}
\left\{ \begin{array}{lll}
\underline I_t - d_I \Delta \underline I  = \beta \left(\frac{N}{|\Omega|} - \varepsilon - \underline I \right)^q \underline I^p - \gamma \underline I,&x\in\Omega,\,t>T,\\ \noalign{\vskip 3pt}
\frac{\partial \underline I}{\partial \nu} =0,&x\in\partial\Omega,\,t>T, \\ \noalign{\vskip 3pt}
\underline I(x,T) = \min_{x\in\overline\Omega}I(x,T),&x\in\Omega,
\end{array}\right.
\label{thm-glo2-6.5}
\end{equation}
In view of the maximum principle for parabolic equations, comparing  \eqref{SIS-mass},  \eqref{thm-glo2-6} and \eqref{thm-glo2-6.5}, we find that
\begin{equation}
\underline I(x,t) \leq I(x,t) \leq \bar I(x,t),\quad (x,t)\in \overline\Omega \times (T,\infty).
\label{thm-glo2-7}
\end{equation}
It follows readily from \eqref{thm-glo2-5} and the fact $S>0$ that
\begin{equation*}
0 < \bar I(x,T) = \max_{x\in\overline\Omega}I(x,T) < \frac{N}{|\Omega|} + \varepsilon.
\label{thm-glo2-7.1}
\end{equation*}
In order to verify that the initial data in \eqref{thm-glo2-6.5} fulfills the requirement in Lemma \ref{glo-lem1} with $\tau_0 = \frac{N}{|\Omega|} - \varepsilon$, let us recall from \cite[Theorem 2.5]{PW19} that there exists some positive constant $\varepsilon_0$ depending only on $N$ such that $\liminf_{t\to\infty}S(x,t) \geq \varepsilon_0$ uniformly for $\overline\Omega$. Thus, by choosing $\varepsilon\in (0, \varepsilon_0)$ small and enlarging $T$ if necessary, this, together with \eqref{thm-glo2-5}, yields
\be
0< \underline I(x,T) = \min_{x\in\overline\Omega}I(x,T) < \frac{N}{|\Omega|} - \varepsilon.
\nonumber
\ee
Thus, we now apply Lemma \ref{glo-lem1} to problems \eqref{thm-glo2-6} and \eqref{thm-glo2-6.5} with $\tau_0 = \frac{N}{|\Omega|} + \varepsilon$ and  $\tau_0 = \frac{N}{|\Omega|} - \varepsilon$, respectively, to conclude the following uniform convergence
\begin{equation}
\lim_{t\to\infty}\underline I(x,t) = U_{\varepsilon}(x) \quad \mbox{and}\quad \lim_{t\to\infty}\bar I(x,t) = U^\varepsilon (x),
\label{thm-glo2-8}
\end{equation}
where $U^\varepsilon$ and $U_\varepsilon$ are the unique positive steady state of \eqref{thm-glo2-6} and \eqref{thm-glo2-6.5}, respectively.

  Recall also from \cite[Theorem 2.5]{PW19} that there exists $\epsilon_1>0$  such that  $\liminf_{t\to\infty}I(x,t) \geq \varepsilon_1$ uniformly for $\overline\Omega$. This allows one to infer there exists some small  $\delta>0$  such that $\left(\delta, \frac{N}{|\Omega|} - \delta\right)$ is a pair of lower-and-upper solutions to  the steady state problem of \eqref{thm-glo2-6} for   any small $\varepsilon>0$. Hence, $0<\delta \leq U^\varepsilon \leq \frac{N}{|\Omega|} - \delta $ for all $x\in\overline\Omega$ by the uniqueness of positive equilibrium recorded in Lemma \ref{glo-lem1}. Then the elliptic regularity estimates yield that $\{U^\varepsilon\}$ is bounded in $C^{2+\alpha}(\overline\Omega)$ for some $\alpha>0$. We then infer from the Ascoli-Arzel\`{a} theorem that, along a subsequence of $\varepsilon\to 0$, it holds $U^\varepsilon \to U^0\geq 0$ in $C^2(\overline\Omega)$.  By sending $\varepsilon \to 0+$, we see that $U^0$ is in fact a positive solution of the elliptic problem
\begin{equation}
\left\{ \begin{array}{ll}
-d_I \Delta u = \beta \left(\frac{N}{|\Omega|} - u\right)^q u^p - \gamma u, &x\in\Omega,\\ \noalign {\vskip 3pt}
\frac{\partial u}{\partial\nu} =0,&x\in\partial\Omega.
\end{array}\right.
\label{thm-glo2-9}
\end{equation}
In the similar fashion, we have $\lim_{\varepsilon\to 0}U_\varepsilon = U_0>0$ in $C^2(\overline\Omega)$, where $U_0$ is also a positive solution of \eqref{thm-glo2-9}. On the other hand, as in the proof of Lemma \ref{glo-lem1}, one can show that the positive solution of \eqref{thm-glo2-9} in $(0,\frac{N}{|\Omega|}]$ is unique. Then it must hold $U_0 \equiv U^0$ on $\overline{\Omega}$.

Finally, by sending $\varepsilon \to 0+$, we conclude from \eqref{thm-glo2-5}, \eqref{thm-glo2-7} and \eqref{thm-glo2-8} that
\begin{equation*}
\lim_{t\to\infty}(S(x,t),I(x,t)) = \left(\frac{N}{|\Omega|}- \tilde I(x),\tilde I(x)  \right)
\end{equation*}
uniformly for $x\in\Omega$, where $\tilde I\equiv U_0\equiv U^0$ is the unique positive solution of \eqref{thm-glo2-9}. Defining $\tilde S(x) = \frac{N}{|\Omega|} - \tilde I(x)>0$ and noting $d_S=d_I$, one can readily verify that $(\tilde S, \tilde I )$ is the unique endemic equilibrium of \eqref{SIS-mass}. The proof is thus completed.
\end{proof}

Now, we consider the case of $p=1$. System \eqref{SIS-mass} has a unique disease-free-equilibrium (DFE) $(\frac{N}{|\Omega|},0)$. According to  \cite{Wang-Zhao}, we define the basic reproduction number $\mathcal R_0$  via
\begin{equation}
\mathcal R_0 = \left(\frac{N}{|\Omega|}\right)^q \sup_{0\neq \varphi\in H^1(\Omega)}\frac{\int_{\Omega}\beta \varphi^2} {\int_{\Omega}\left(d_I|\nabla \varphi|^2 + \gamma \varphi^2 \right)}.
\label{R0}
\end{equation}
We here point out that $\mathcal R_0<1$ when $ \lambda_*>0$,  $\mathcal R_0=1$ when $ \lambda_*=0$  and $\mathcal R_0>1$ when $ \lambda_*<0$. Here and below, $\lambda_*$ is the principal eigenvalue of the eigenvalue problem
\begin{equation}
\left\{ \begin{array}{ll}
-d_I \Delta \phi + \left(\gamma - \beta \left(\frac{N}{|\Omega|}\right)^q  \right)\phi = \lambda \phi,&x\in\Omega,\\ \noalign {\vskip 3pt}
\frac{\partial \phi}{\partial \nu}=0,&x\in\partial\Omega.
\end{array}\right.
\label{eigen-prob}
\end{equation}
The DFE $(\frac{N}{|\Omega|},0)$ is locally asymptotically stable if $\mathcal R_0<1$, while it is unstable if $\mathcal R_0>1.$

Now, suppose   there exists some $r>0$ such that $\gamma(x) = r \beta(x)$ for all $x\in\overline\Omega$. It is not hard to see from \eqref{R0} that $\mathcal R_0 = \frac{1}{r}\left(\frac{N}{|\Omega|}\right)^q$. If $\mathcal R_0>1$, a constant EE $(\hat S, \hat I):=(r^{\frac{1}{q}},\frac{N}{|\Omega|}-r^{\frac{1}{q}})$ exists.

Concerning the large-time dynamics of \eqref{SIS-mass}, we have the following result.
\begin{theorem}
\label{thm-glo3}
 Let $p=1$ and let  $r>0$ be such that $\gamma(x) = r \beta(x)$ for all $x\in\overline\Omega$. Then the reaction-diffusion SIS model \eqref{SIS-mass} admits  threshold dynamics. More precisely, the unique classical solution $(S,I)$ of \eqref{SIS-mass} enjoys the following convergence properties:
\begin{itemize}
\item[\rm (i)] If $\mathcal R_0\leq 1$, then as $t\to\infty$, the following uniform convergence  on $\overline\Omega$ holds:
\be
\left(S(\cdot,t),I(\cdot,t)\right) \to \left(\frac{N}{|\Omega|},0\right).
\label{thm-glo3-res1}
\ee
\item[\rm (ii)] If $\mathcal R_0>1$, then as $t\to\infty$, the following uniform convergence  on $\overline\Omega$ holds:
\be
\left(S(\cdot,t),I(\cdot,t)\right) \to \left(r^{\frac{1}{q}},\frac{N}{|\Omega|}-r^{\frac{1}{q}}\right).
\label{thm-glo3-res2}
\ee
\end{itemize}
\end{theorem}

\begin{proof}
(i) Since $\mathcal R_0\leq 1$ implies $r^{\frac{1}{q}}\geq \frac{N}{|\Omega|}$, we define the Lyapunov functional
\begin{equation*}
V_3(S,I) = \int_\Omega \left[\frac{1}{2}\left(S- \frac{N}{|\Omega|}\right)^2 + \left(r^{\frac{1}{q}} - \frac{N}{|\Omega|}\right) I \right].
\end{equation*}
Then we compute from \eqref{SIS-mass} that
\begin{align*}
\frac{d}{dt}V_3(S,I)& = \int_\Omega \left(S- \frac{N}{|\Omega|}\right)(d_S \Delta S -\beta S^q I + r\beta I) \\
& \ + \left(r^{\frac{1}{q}} - \frac{N}{|\Omega|}\right) \int_\Omega (d_I \Delta I + \beta S^q I - r \beta I)\\
& = -d_S \int_\Omega |\nabla S|^2 - \int_\Omega \beta I \left(S^q - r \right) \left(S - r^{\frac{1}{q}}\right) \leq 0,
\end{align*}
which  upon integration and the boundedness of $(S,I)$  implies
\be
\nonumber
%\label{thm-glo3-1}
\int_1^\infty \int_\Omega \left[d_S |\nabla S|^2+ \beta I \left(S^q-r\right)\left(S- r^\frac{1}{q}\right) \right]<\infty.
\ee
The regularity in \eqref{reg-SI} along with Ascoli-Arzel\`{a} theorem implies there exists $t_k\rightarrow \infty$ and nonnegative $C^2$-function $(S_\infty, I_\infty)$ such that $(S(\cdot, t_k), I(\cdot, t_k))\rightarrow (S_\infty,I_\infty)$ as $k\rightarrow \infty$ in $C^2(\overline{\Omega})$. On the other hand, since the function
$$t \mapsto \int_\Omega \left[d_S |\nabla S|^2+ \beta I \left(S^q-r\right)\left(S - r^\frac{1}{q}\right) \right]$$
 is  uniformly bounded and uniformly continuous, we infer  that
\be
\label{thm-glo3-2}
\int_\Omega|\nabla S_\infty|^2 + \int_\Omega \beta I_\infty \left(S_\infty^q- r \right)
\left(S_\infty - r^\frac{1}{q}\right) =0,
\ee
which entails either $S_\infty\equiv r^{1/q}$ or $I_\infty\equiv 0$. If $S_\infty\equiv r^{1/q}$, then we infer from the conservation of total population \eqref{N} that $I_\infty = \frac{N}{|\Omega|} - r^{\frac{1}{q}}$, which is non-positive since $\mathcal R_0\leq 1$. Thus, we must have $I_\infty \equiv 0$ and consequently $S_\infty \equiv \frac{N}{|\Omega|}$.  Since this type of argument can be applied to every subsequence of $\{t\}_{t\geq1}$, \eqref{thm-glo3-res1} follows readily.

(ii) Recall that $(\hat S,\hat I):=(r^{\frac{1}{q}},\frac{N}{|\Omega|}-r^{\frac{1}{q}})$ is an EE. Define
\be
V_4(S,I) = \int_\Omega \left[\frac{1}{2}\left(S- \hat S +I -\hat I \right)^2 + \frac{(d_S + d_I)^2}{8d_S d_I}\left(S - \hat S\right)^2   \right].
\nonumber
\ee
Then
\begin{align}
\frac{d}{dt}V_4(S,I) & =\int_\Omega \left(S- \hat S + I -\hat I\right) \left(d_S \Delta S + d_I \Delta I\right) \nonumber\\
&\quad \ + \frac{(d_S + d_I)^2}{4d_S d_I} \int_\Omega \left(S- \hat S\right) \left(d_S \Delta S - \beta S^qI + r\beta I\right) \nonumber\\
& = -d_S \int_\Omega |\nabla S|^2 -\int_\Omega \left|\frac{d_S +d_I}{2 \sqrt{d_I}}\nabla S + \sqrt{d_I}\nabla I \right|^2 \nonumber\\
& \quad \ -\frac{(d_S + d_I)^2}{4d_S d_I}\int_\Omega \beta I \left(S- \hat S\right) \left(S^q - \hat S^q \right)  \leq 0,\nonumber
\end{align}
which implies that \eqref{thm-glo3-2} holds for some limiting  $C^2$-nonnegative functions  $(S_\infty,I_\infty)$ of $(S,I)$.
Hence, as argued above, we see that
\be\label{limt-choice}
(S_\infty,I_\infty)\in\left\{\left(\frac{N}{|\Omega|}, 0\right), \ \left(\hat S,\hat I\right) \right\}.
\ee
(Note that the limiting functions must be unique since $\lim_{t\to\infty}V_4(S,I)$ exists and $V_4(\frac{N}{|\Omega|},0) \neq V_4(\hat S,\hat I)$.)
Since $\mathcal R_0>1$, the former case can be eliminated by the uniform persistence property of $(S,I)$ in \cite[Theorem 2.5]{PW19}, which is obtained by abstract dynamical system theory. Nevertheless, here we would like to present an elementary argument to rule out this possibility. In fact, suppose $(S_\infty, I_\infty) = (\frac{N}{|\Omega|},0)$. Observe $S_\infty = \frac{N}{|\Omega|} > r^{\frac{1}{q}}$ since $\mathcal R_0>1$. Then there exists some positive constant $\varepsilon_0$ and $T>0$ such that
$$
S^q - r \geq \varepsilon_0>0\ \ \text{on}\ \  \overline{\Omega} \times[T, \infty).
$$
Hence, by the $I$-equation in \eqref{SIS-mass}, we find that
\begin{equation*}
\begin{cases}
I_t\geq d_I\Delta I + \varepsilon_0 \left(\min_{\overline\Omega}\beta\right) I,& x\in\Omega, \, t>T,\\[3pt]
\frac{\partial I}{\partial \nu}=0, & x\in \partial \Omega, \, t>T,\\[3pt]
 I(x,T)\geq  \min_{x\in\overline\Omega} I(x, T), & x\in \Omega,
\end{cases}
\end{equation*}
which upon comparison with the corresponding ODE shows, for $(x,t)\in \overline{\Omega}\times[T,\infty)$, that
$$
I(x,t)\geq \min_{x\in\overline{\Omega}}I(x,T)\exp\left\{\varepsilon_0 \left(\min_{\overline\Omega}\beta\right) \left(t-T\right)\right\}.
$$
This clearly is incompatible with the boundedness of $I$. Then the desired uniform convergence in \eqref{thm-glo3-res2} follows from \eqref{limt-choice}.
\end{proof}

Finally, we  turn our attention to the scenario of equal diffusion rates.

\begin{theorem}
\label{thm-glo4}
Suppose $p=1$ and $d_S =d_I$. Then the reaction-diffusion SIS model \eqref{SIS-mass} has the following   threshold dynamics:
\begin{itemize}
\item[\rm (i)] If $\mathcal R_0\leq 1$, then as $t\to\infty$,  the following uniform convergence  on $\overline\Omega$ holds:
\be
\left(S(\cdot,t),I(\cdot,t)\right) \to \left(\frac{N}{|\Omega|},0\right).
\label{thm-glo4-res1}
\ee
\item[\rm (ii)] If $\mathcal R_0>1$, then as $t\to\infty$,  the following uniform convergence  on $\overline\Omega$ holds:
\be
\left(S(\cdot,t),I(\cdot,t)\right) \to \left(\tilde S, \tilde I\right),
%\label{thm-glo4-res2}
\nonumber
\ee
 where $(\tilde S,\tilde I)$ is the unique EE of \eqref{SIS-mass}.
\end{itemize}
\end{theorem}

\begin{proof}
(i) Firstly, suppose $\mathcal R_0<1$. Then $\lambda_*>0$, where $\lambda_*$ is the principal eigenvalue of \eqref{eigen-prob}. As in the reasoning in Theorem \ref{thm-glo2}, we are led to the auxliary problems \eqref{thm-glo2-6} and \eqref{thm-glo2-6.5} (with $p=1$). To analyze the large-time dynamics of the unique solution to \eqref{thm-glo2-6}, for any small $\varepsilon>0$,  we consider the eigenvalue problem
\begin{equation}
\left\{ \begin{array}{ll}
-d_I \Delta \phi + \left(\gamma - \beta \left(\frac{N}{|\Omega|} + \varepsilon \right)^q  \right)\phi = \lambda \phi,&x\in\Omega,\\ \noalign {\vskip 3pt}
\frac{\partial \phi}{\partial \nu}=0,&x\in\partial\Omega,
\end{array}\right.
%\label{eigen-prob2}
\nonumber
\end{equation}
whose principal eigenvalue is denoted by $\lambda_\varepsilon$. It follows from $\lambda_*>0$ that $\lambda_\varepsilon>0$ for small $\varepsilon>0$. We then claim that \eqref{thm-glo2-6} does not admit non-negative non-trivial steady state in the interval $[0,\frac{N}{|\Omega|}+\varepsilon]$. In fact, suppose $I_\varepsilon \geq 0,\, \not\equiv0$ is such a steady state. Then we derive from the variational formulation of $\lambda_\varepsilon$ and the equation satisfied by $I_\varepsilon$ that
\be\label{thm-glo4-0}
\begin{split}
\lambda_\varepsilon &= \inf_{0\neq \phi\in H^1(\Omega)} \frac{\int_\Omega \left[ d_I |\nabla \phi^2| + \left(\gamma - \beta \left(\frac{N}{|\Omega|}+\varepsilon \right)^q  \right) \phi^2 \right]}  {\int_\Omega \phi^2} \\
& \leq \frac{\int_\Omega \left[ d_I |\nabla I_\varepsilon^2| + \left(\gamma - \beta \left(\frac{N}{|\Omega|}+\varepsilon \right)^q  \right) I_\varepsilon^2 \right]}  {\int_\Omega I_\varepsilon^2} \\
& = \frac{\int_\Omega \beta I_\varepsilon^2 \left[\left(\frac{N}{|\Omega|}+\varepsilon - I_\varepsilon \right)^q  - \left(\frac{N}{|\Omega|}+\varepsilon \right)^q  \right]}  {\int_\Omega I_\varepsilon^2} \\
&<0,
\end{split}
\ee
contradicting $\lambda_\varepsilon>0$. As a result, $\lim_{t\to\infty}\bar I(\cdot,t) =0$ uniformly on $\overline \Omega$. Then we derive from \eqref{thm-glo2-7} that $\lim_{t\to\infty}I(\cdot,t) =0$ and therefore $\lim_{t\to\infty}S(\cdot,t) = \frac{N}{|\Omega|}$ uniformly on $\overline\Omega$.

Now suppose $\mathcal R_0=1$ or equivalently  $\lambda_*=0$. This coupled with \eqref{thm-glo4-0} shows  $\lambda_\varepsilon <0$ for small $\varepsilon>0$. Let $\phi_\varepsilon>0$ be the corresponding principal eigenfunction, normalized by $\|\phi_\varepsilon\|_{L^2} =1$. For the steady state problem of \eqref{thm-glo2-6}, one can choose small $\tau>0$ and use $\tau \phi_\varepsilon$ as a lower solution and $\frac{N}{|\Omega|} + \varepsilon$ as an upper solution to establish the existence of a positive solution $I^\varepsilon$ in $(0,\frac{N}{|\Omega|}+\varepsilon)$. The uniqueness and global asymptotic stability can be proved similarly as in the argument of Lemma \ref{glo-lem1}. Thus, we have
\be
 0 \leq \limsup_{t\to\infty}I(\cdot,t) \leq \lim_{t\to\infty}\bar I(\cdot,t) = I^\varepsilon(\cdot),
\label{thm-glo4-1}
\ee
uniformly on $\overline\Omega$. On the other hand, notice that $\{I^\varepsilon\}$ is uniformly bounded for all small $\varepsilon>0$. As before, along a subsequence of $\varepsilon\to0$, it holds $I^\varepsilon\to I^0 \geq 0$ in $C^2(\overline\Omega)$. The fact  $\lambda_*=0$ enforces that the limiting function $I^0$ must be zero, since otherwise one can easily derive a contradiction as in \eqref{thm-glo4-0}. Thus, \eqref{thm-glo4-res1} follows from \eqref{thm-glo4-1} and \eqref{thm-glo2-5}.

(ii) Since $\mathcal R_0>1$, we have $\lambda_*<0$ and $\lambda_\varepsilon<0$ for all small $\varepsilon>0$. Then both \eqref{thm-glo2-6} and \eqref{thm-glo2-6.5} admit a unique positive equilibrium which is globally asymptotically stable. The rest of the argument is quite similar to that of Theorem \ref{thm-glo2} and hence is omitted.
\end{proof}

\begin{remark}
\label{gs-rk}
For model \eqref{SIS-mass} with $0<p<1$, our findings in Theorems {\rm \ref{thm-glo1}} and {\rm \ref{thm-glo2}} strengthen the results in \cite[Theorem 2.5]{PW19}, where only uniform persistence property of solutions is shown. For \eqref{SIS-mass} with $p=1$, the global stability results in Theorems {\rm \ref{thm-glo3}} and {\rm \ref{thm-glo4}} confirm the conjecture in \cite[Remark 2.3 (i)]{PW19} and thus extends \cite[Theorem 4.1]{DW16}.
\end{remark}

{\small
\section*{Acknowledgment}

H. Li was partially supported by National Key R\&D Program of China (No. 2021YFA1002100), NSF of China (No. 11971498), the project of Guangzhou Science, Technology and Innovation Commission (No. 202102020807).  T. Xiang was partially supported by NSF of China (Nos. 12071476  and 11871226).

The authors would like to thank Prof. Rui Peng from Jiangsu Normal University, China for many useful discussions.
}

\end{document}